\documentclass{amsart}
\usepackage{a4wide}
\usepackage[dvips]{graphicx}
\usepackage{color}
\usepackage{amssymb}
\usepackage{textcomp}

\usepackage{hyperref}
\usepackage{amsmath}
\usepackage{amsfonts}
\usepackage{latexsym}
\usepackage{amscd}
\usepackage{graphicx}
\usepackage[latin1]{inputenc}
\usepackage[english]{babel}
\usepackage{enumerate}
\usepackage{afterpage}
\usepackage{cleveref}

\let\oldtocsection=\tocsection

\let\oldtocsubsection=\tocsubsection

\let\oldtocsubsubsection=\tocsubsubsection

\renewcommand{\tocsection}[2]{\hspace{0em}\oldtocsection{#1}{#2}}
\renewcommand{\tocsubsection}[2]{\hspace{1em}\oldtocsubsection{#1}{#2}}
\renewcommand{\tocsubsubsection}[2]{\hspace{2em}\oldtocsubsubsection{#1}{#2}}

\def\a{\alpha}
\def\b{\beta}
\def\d{\delta}
\def\e{\epsilon}
\def\g{\gamma}

\newcommand{\N}{\mathbb{N}}

\renewcommand{\to}{\longrightarrow}
\def\co{\colon\thinspace}

\newcommand{\C}{\mathbb{C}}
\newcommand{\Z}{\mathbb{Z}}

\newcommand{\R}{\mathbb{R}}

\newcommand{\SL}{\mathrm{SL}}

\newcommand{\SU}{\mathrm{SU}}

\newcommand{\Tr}{\mbox{tr }}

\newcommand{\MCG}{\mathcal{MCG}}

\newcommand{\HH}{{\mathbf H}^2}

\newcommand{\X}{\mathfrak{X}}

\newcommand{\Sc}{\mathcal{S}}
\newcommand{\Ec}{\mathcal{E}}
\newcommand{\tr}{\mathrm{tr}}
\newcommand{\hm}{\mathrm{Hom}}

\renewcommand{\to}{\longrightarrow}
\def\co{\colon\thinspace}

\newcounter{notes}

\def\a{\alpha}
\def\b{\beta}
\def\d{\delta}
\def\e{\epsilon}
\def\g{\gamma}

\renewcommand{\to}{\longrightarrow}
\def\co{\colon\thinspace}

\newcommand{\Tc}{\mathcal{T}}

\newcommand{\Out}{\mathrm{Out}}

\newcommand{\SUtwo}{\mathrm{SU}(2,1)}
\newcommand{\SLthreeC}{\mathrm{SL}(3,\C)}

\theoremstyle{plain}
\newtheorem{Theorem}{Theorem}[section]
\newtheorem{Lemma}[Theorem]{Lemma}
\newtheorem{Proposition}[Theorem]{Proposition}
\newtheorem{Corollary}[Theorem]{Corollary}
\newtheorem{introthm}{Theorem}

\theoremstyle{definition}
\newtheorem{Definition}[Theorem]{Definition}
\newtheorem{Remark}[Theorem]{Remark}

\begin{document}

\title{Dynamics on the $\SUtwo$--character variety of the one-holed torus}

\author{Sean Lawton}
\address{Department of Mathematical Sciences, George Mason University}
\email{slawton3@gmu.edu}
\urladdr{https://science.gmu.edu/directory/sean-lawton}

\author{Sara Maloni}
\address{Department of Mathematics, University of Virginia}
\email{sm4cw@virginia.edu}
\urladdr{https://sites.google.com/view/sara-maloni}

\author{Fr\'{e}d\'{e}ric Palesi}
\address{I2M, Aix Marseille Universit\'{e}, CNRS, UMR 7373, 13453 Marseille, France}
\email{frederic.palesi@univ-amu.fr}
\urladdr{www.i2m.univ-amu.fr/perso/frederic.palesi}

\begin{abstract}
We study the relative $\SUtwo$--character varieties of the one-holed torus, and the action of the mapping class group on them. We use an explicit description of the character variety of the free group of rank two in $\SUtwo$ in terms of traces, which allow us to describe the topology of the character variety. We then combine this description with a generalization of the Farey graph adapted to this new combinatorial setting, using ideas introduced by Bowditch. Using these tools, we can describe an open domain of discontinuity for the action of the mapping class group which strictly contains the set of convex cocompact characters, and we give several characterizations of representations in this domain. 
\end{abstract}



\maketitle

\tableofcontents

\section{Introduction}\label{s:intro}

Let $\Gamma$ be the fundamental group of the $1$--holed torus $S_{1,1}$, that is, $\Gamma \cong F_2 =\langle \a, \b\rangle$, a free group of rank $2$. The Lie group $\SU(2,1)$ acts on the space of group homomorphisms $\hm(F_2,\SU(2,1))$ by conjugation.  Let $\hm(F_2,\SU(2,1))^*$ be the subspace of $\hm(F_2,\SU(2,1))$ consisting of homomorphisms having closed conjugation orbits, and define the $\SU(2,1)$-character variety of $F_2$ to be the quotient space $\X = \hm(F_2, \SU(2,1))^*/\SU(2,1)$.  The mapping class group $\MCG(S_{1,1})$ is isomorphic to the outer automorphisms of $\Gamma$ and acts on $\X$ by pre-composition.

The main result of this paper is the following: 
\begin{introthm}\label{Main} There exists an open domain of discontinuity $\X_{BQ}$ for the action of the mapping class group $\MCG(S_{1,1})$ on $\X$ which strictly contains the set of discrete, faithful, convex cocompact representations.
\end{introthm}

To prove this theorem we combine ideas of Bowditch \cite{bow_mar} (further generalized by Tan-Wong-Zhang \cite{tan_gen}, Maloni-Palesi-Tan \cite{mal_ont} and Maloni-Palesi \cite{MaPa1}) with the explicit parametrization of the $\SLthreeC$-character variety of $F_2$ described by Lawton \cite{law0,law1}. In order to describe the set of representations that form $\X_{BQ}$, called \emph{Bowditch representations}, we need to analyze what are called ``relative character varieties." Given $c \in \C$, we define the {\it relative character variety} $\X_c$ to be the set of $\SU(2,1)$-valued representations of $\Gamma$ with fixed trace of the peripheral element $[\a, \b]$:
 $$\X_c =\{[\rho] \in \X \mid \mathrm{tr}(\rho([\a, \b])) = c \}.$$

As a bi-product of our work, we can show that Bowditch representations can be defined in many equivalent ways. For an isometry $\varphi$ of complex hyperbolic space $ \mathbf{H}_\C^2$, define  $l(\varphi)$ as follows: 
$$l(\varphi) := \inf \{ d_{ \mathbf{H}_\C^2}(x , \varphi (x)) \mid x \in \mathbf{H}_\C^2 \}.$$

Let $\Sc$ be the set of free homotopy classes of essential unoriented simple closed curves on $S_{1,1}$. We can define the following properties for representations in $\X_c$:
\begin{enumerate}
  \item[{(BQ1)}] $\forall \g \in \Sc$, $f(\Tr\rho(\g)) >0$ (equivalently $\rho(\g)$ is loxodromic), where $$f(t)=|t|^4 - 8 \mathrm{Re} (t^3) + 18 |t|^2 - 27.$$ 
  \item[{(BQ2)}] $\#\{\g\in\Sc \mid |\Tr\rho(\g)| \le M(c) \} < \infty$, where $M(c)> 0$ is defined in Section \ref{definition}.
  \item[{(BQ2')}] $ \forall K>0$, $\; \#\{\g\in\Sc \mid |\Tr\rho(\g)| \le K\} < \infty.$
  \item[{(BQ3)}] There exists $k, m >0$ such that $|l (\rho (\gamma))| \geq k W(\gamma) - m $ for all $\gamma \in \Sc$.
  \item[{(BQ4)}] $\forall K\geq M(c)$, the attracting subgraph $T_\rho (K)$ (defined in Section \ref{sec:tree}) is finite.
  \end{enumerate}
  
In Section \ref{sec:bowditch-characterization} we prove the following result:
\begin{introthm}\label{Main-characterization} 
Let $[\rho]\in\X_c$. The following are equivalent:
\begin{enumerate}
\item[$(1)$] The representation $\rho$ satisfies $(BQ1)$ and $(BQ2)$;
\item[$(2)$] The representation $\rho$ satisfies $(BQ1)$ and $(BQ2')$;
\item[$(3)$] The representation $\rho$ satisfies $(BQ3)$;
\item[$(4)$] The representation $\rho$ satisfies $(BQ4)$.
\end{enumerate}
\end{introthm}

So the set of Bowditch representations $\X_{BQ}$ referred to in Theorem \ref{Main} is the union over $c\in \C$ of the Bowditch representations in $\X_c$ defined by any of the equivalent conditions in Theorem \ref{Main-characterization}.  A corollary of the above theorem, given the work of \cite{schlich}, is that our set $\X_{BQ}$ coincides with the set of ``primitive-stable representations" defined in \cite{kim_kim_2020}.

We also prove the somewhat surprising result that the topology of the character variety $\X$ is non-trivial.  In particular, we prove:
\begin{introthm}\label{Main2}
The $\SU(2,1)$-character variety of $F_2$ $\X$ strong deformation retracts onto $S^1\times S^1$.  In particular, it is not simply connected.
\end{introthm}
This is surprising because the $\SU(3)$-character variety of $F_2$ satisfies the same real algebraic equations and has the same (real) dimension, but it is not homeomorphic to $\X$, since the $\SU(3)$-character variety of $F_2$ is homeomorphic to an $8$-sphere by Florentino and Lawton \cite{FlLa1}.

\subsection*{Proof Strategy} 

We now give an overview of the proof of \Cref{Main} and \Cref{Main-characterization}. 
The proof can be organized into five main steps:
\begin{enumerate}
	\item[(0)] Description of the combinatorial setup: $\mathcal{C}, \mathcal{T},\mathcal{E}$;
	\item[(1)] Orientation on $\mathcal{E}$, fork lemma and connectivity of the regions $\Sigma$;
	\item[(2)] Growth of traces for neighbors around a region and for escaping rays, and definition of an attracting subgraph of $\mathcal{E}$;
	\item[(3)] Fibonacci growth and equivalent characterizations of Bowditch representations;
	\item[(4)] Proof that the set $\X_{BQ}$ strictly contains convex cocompact representations. 
\end{enumerate}

\begin{enumerate}
\item[Step (0):] We start with a discussion of the combinatorial setup needed for our study. Given the one-hole torus $S:=S_{1,1}$, we define three simplicial complexes associated to $S$: the curve complex $\mathcal{C}$, its dual $\mathcal{T},$ and the edge graph $\mathcal{E}$. The curve complex $\mathcal{C}$ is the $2$--dimensional simplicial complex, where $k$--simplices are given by $k+1$ distinct (free homotopy classes of) simple closed curves in $S$ that pairwise intersect once. Its simplicial dual $\mathcal{T}$ is a countably infinite tri-valent simplicial tree properly embedded in $\HH$ defined by the property that the set $\mathcal{T}^{(k)}$ of $k$--simplices in $\mathcal{T}$ satisfies $\mathcal{T}^{(k)} = \mathcal{C}^{(2-k)}$ for $k =0,1,2$. Finally, the \textit{edge graph} $\mathcal{E}$ has vertices corresponding to edges of $\mathcal{T}$, and edges corresponding to adjacency, so one can see that $\mathcal{E}$ is a ``trivalent tree of triangles" in the sense that it is a graph made of triangles (one for each triangle $T$ in $\mathcal{C}^{(2)}$ or, equivalently, one for each vertex of $\mathcal T$). See Figure \ref{fig:f-t}.

\item[Step (1):] For the first step, we show how any representation in $\X_c$ induces an orientation on the edges of $\mathcal{E}$ as follows. An edge $(X, Y, Z; T, T) \in \mathcal{E}$ indicates that the edge is in the triangle $(X, Y, Z) \in \mathcal{C}^2$ and is between the triangles $(X, Y, T), (X, Z, T') \in \mathcal{C}^2$ which are adjacent to $(X, Y, Z)$, see Figure \ref{fig:e_oriented}. If $t = |\tr (\rho (T)) | > |\tr (\rho (T'))| = t' $, then the oriented edge points towards $T'$. In case where there is an equality, one can choose the orientation arbitrarily (and the choice will not affect the result). 

We have a $4$--coloring of the edges of $\mathcal{E}$ such that at each vertex there are only two triangles of two different colors meeting. We say that a vertex $v \in \mathcal{E}^{(0)}$ is a \textit{fork} if there are (at least) two arrows of different colors pointing away from $v$. The Fork Lemma (\Cref{fork}) shows that if a vertex $v = (X, Y) \in \mathcal{E}^{(0)}$ is a fork, then $\min \{|x| , |y| \} \leq \max\{6,\sqrt{|\Re(c)|}\}$ for $x=\Tr(\rho(X))$ and $y=\Tr(\rho(Y))$. We use this result to show, for $[\rho] \in \X_c$, the set $\Omega_\rho (K)$  of regions with traces smaller than $K \geq  \max\{6,\sqrt{|\Re(c)|}\}$ is connected. The proof is by contradiction and considers the distance between two connected components. The Fork Lemma is used to reach the conclusion when such a distance is greater than $2$ since in that case a fork will appear.

\item[Step (2)] For the second step we first study the behavior of the values of the (norm of the) traces of regions that are all neighbors of a central region $X$. Given $[\rho] \in \X_c$ we see that if we denote $(X_n)_{n \in \Z}$ the regions around $X$, then the behavior of the sequence $(u_n)_{n\in \N}$ defined by $u_n = \tr (\rho (X_n))$ depends on the value of $f(x)$ where $x=\Tr(\rho(X))$ and $f(x) = \mathrm{Res} (\chi_x , \chi_x')$. In fact if $f(x) = 0$, then $(u_n)_{n\in \N}$ grows at most quadratically, while if $f(x) < 0$, then $(u_n)_{n\in \N}$ remains bounded. Finally, in \Cref{prop:bounded_neighbors} we show that there exists a constant $D = D(c)$, such that if the sequence $(u_n)_{n\in \N}$ defined above with $f(x) >0$ does not grow exponentially in both directions, then there are infinitely many terms of the sequence such that $|u_n| < D$. We then consider \emph{escaping rays}, that is infinite geodesic rays where each edge $e_n$ is directed from $v_n$ to $v_{n+1}$. In \Cref{lem:escaping} show that if $\{ e_n \}_{n\in \N}$ is an escaping ray, then there are two cases: 
\begin{enumerate}
\item[$(1)$] there exists a region $\a$ such that the ray is eventually contained in the boundary of $\a$, such that $f(x) \leq 0$ with $x = \tr (\rho (\a))$, or 
\item[$(2)$] the ray meets infinitely many distinct regions with trace smaller than $M(c)$, where $M(c):=\max \{6, \sqrt{|\Re (c)|} , D(c)\}$.
\end{enumerate}

Using the description of the growth of the traces for regions around a fixed central trace, we define, for every representation $[\rho] \in \X_c$ and for every $K> M(c)$, a connected attracting subgraph $T_\rho(K)$ of $\mathcal{E}$. (Recall that $\mathcal{E}$ is a ``trivalent tree of triangles", so we want to define $T_\rho(K)$ so that it is a union of ``triangles''.) The connectivity of $T_\rho(K)$ comes from the connectivity of $\Omega_\rho (K)$, while the fact that $T_\rho(K)$ is attractive comes from the definition of $T_\rho(K)$. We then show that for representations in $\X_{BQ}$, that is representations satisfying conditions (BQ1) and (BQ2), such a graph $T_\rho(K)$ is finite for all $K> M(c)$.

\item[Step (3):] In the penultimate step, given a vertex $v \in \mathcal{E}^{(0)}$ we define a Fibonacci function $F_v : \Omega \rightarrow \N$. We then show that this function has the following property: if $\{\alpha, \beta\}$ is a set of free generators for $F_2$ corresponding to regions $X$ and $Y$ such that $v =(X, Y) \in \mathcal{E}^{(0)}$, then for any element $\gamma \in \Omega$ the Fibonacci function corresponds to the word length of $\gamma$ with respect to $\{\alpha, \beta\}$, that is $F_v(\gamma) = W(\gamma)$. Note that we only care about the asymptotic growth of this function, as you will see, so the choice of $v$ will not affect our results. We then define what it means for a function $g \co\Omega \rightarrow [0, \infty)$ to have \textit{Fibonacci growth}: if there exist constants $\kappa_1, \kappa_2 >0$ such that $$\kappa_1 F_v (X) \leq g(X) \leq \kappa_2 F_v (X)$$ for all $X \in \Omega$. If only the lower (respectively upper) bound is satisfied, the function $g$ is said to have a lower (respectively upper) Fibonacci growth. Given any $[\rho] \in \mathfrak{X}$, we denote by $\phi_\rho \co\Omega \rightarrow \C$ the function $\phi_\rho (X) := \log | \tr (\rho (X)) |$. We consider the function $\phi_\rho^+ := \max \{ \phi_\rho , 0 \}$. We then show that this function always has an upper Fibonacci growth, but for representations in $[\rho] \in \mathfrak{X}_{BQ}$ we show that the function has upper and lower Fibonacci growth. 

We use this result to show that we can characterize representations in $\X_{BQ}$ in terms of the attracting graph $T_\rho(K)$ as follows: $[\rho]$ is in $(\X_c)_{BQ}$ if and only if the subgraph $T_\rho (K)$ is finite for all $K> M(c)$, which proves \Cref{Main-characterization}.  This result is crucial to show that the set $\X_{BQ}$ is open and the mapping class group acts on it properly discontinuously, and hence to prove \Cref{Main}.

\item[Step (4):] In order to show that the set $\X_{BQ}$ strictly contains convex cocompact representations, we use representations described by Will \cite{Will-07, Will-12}: a family of discrete, faithful and type-preserving representations of the once-punctured torus in $\X(F_2, \SU(2,1))$. These representations will take non-peripheral, non-trivial simple closed curves to loxodromic elements and will map the commutator to a parabolic unipotent element. In order to see that these representations are in $\X_{BQ}$, we use a characterization of both sets in term of the growth of $|l (\rho (\gamma))|$: the existence of $k, m >0$ such that $|l (\rho (\gamma))| \geq k W(\gamma) - m $ for all $\gamma \in F_2$ corresponding to non-peripheral, non-trivial simple closed curves in $S_{1,1}$. On the other hand, since the (peripheral) simple closed curve corresponding to the commutator is mapped to a unipotent element, we can see these representations are not convex cocompact. There are also other examples of representations that are in $\X_{BQ} \setminus \X_{CC}$ and such that the peripheral element is mapped to a parabolic non-unipotent element, see Falbel-Parker \cite{Falbel-Parker}, Falbel-Koseleff \cite{Falbel-Kos-rig, Falbel-Kos-circ}, or Parker-Gusevskii \cite{Gus-Parker-03, Gus-Parker-00}. See Will \cite{Will-16} for a discussion of these references. 
\end{enumerate}

\subsection*{General Setting}
Let $\Gamma$ be a finitely presentable group with the discrete topology and let $G$ be a Lie group.  The set $\mathrm{Hom}(\Gamma,G)$ of homomorphisms from $\Gamma$ to $G$ is a topological space by giving it the compact-open topology. The group $G$ acts on $\mathrm{Hom}(\Gamma,G)$ by conjugation, that is, $g\cdot\rho:= \iota_g\circ \rho$ where $\iota_g\in \mathrm{Inn}(G)$.  Let $\mathrm{Orb}_G(\rho)$ be the $G$-orbit of $\rho$.  Define the subspace of {\it polystable} homomorphisms $$\mathrm{Hom}(\Gamma,G)^*:=\{\rho\in\mathrm{Hom}(\Gamma,G)\ |\ \mathrm{Orb}_G(\rho)=\overline{\mathrm{Orb}_G(\rho)}\},$$ and define the {\it polystable} quotient space $$\mathfrak{X}(\Gamma,G):=\mathrm{Hom}(\Gamma,G)^*/G.$$  This quotient is called the $G$--{\it character variety of} $\Gamma$ for historical reasons, even though in this generality it may be neither a variety, see Casimiro-Florentino-Lawton-Oliverira \cite{CFLO}, nor correspond to characters, see Lawton-Sikora \cite{LaSi}. 

Let $\mathcal{P}$ be a collection of properties such that if $\rho\in \mathrm{Hom}(\Gamma, G)$ satisfies $\mathcal{P}$, then $\iota_g\circ \rho$ also satisfies $\mathcal{P}$ for all $g\in G$.  Then $G$ acts on the subspace $$\mathrm{Hom}_\mathcal{P}(\Gamma, G):=\{\rho\in \mathrm{Hom}(\Gamma, G)\ |\ \rho \ \mathrm{ satisfies }\ \mathcal{P}\}$$ by conjugation.  The polystable quotient of $\mathrm{Hom}_\mathcal{P}(\Gamma, G)$ by $G$, denoted $\mathfrak{X}_\mathcal{P}(\Gamma, G)$, is called the {\it $\mathcal{P}$-relative $G$-character variety of $\Gamma$}.  If $\mathcal{P}$ is vacuous, one gets back $\mathfrak{X}(\Gamma, G).$ 

Let $\mathrm{Aut}_\mathcal{P}(\Gamma)$ denote the subgroup of $\mathrm{Aut}(\Gamma)$ that preserves $\mathrm{Hom}_\mathcal{P}(\Gamma, G)$.  Then $\mathrm{Inn}(\Gamma)\subset \mathrm{Aut}_\mathcal{P}(\Gamma)$ and so $\mathrm{Out}_\mathcal{P}(\Gamma):=\mathrm{Aut}_\mathcal{P}(\Gamma)/\mathrm{Inn}(\Gamma)$ acts on $\mathfrak{X}_\mathcal{P}(\Gamma, G)$ by $[\alpha]\cdot[\rho]=[\rho\circ \alpha^{-1}]$.  Again, if $\mathcal{P}=\emptyset$, one obtains an action of $\mathrm{Out}(\Gamma)$ on $\mathfrak{X}(\Gamma, G)$.

One of the main theorems in Rapinchuk \cite{Rap} says that for any complex affine variety $V$, there exists a finitely generated group $\Gamma$ and a complex reductive group $G$, so that $\mathfrak{X}(\Gamma, G)-\{\chi_0\}$ and $V$ are isomorphic where $\chi_0$ is the point in the character variety corresponding to the trivial representation.  Let $\mathcal{F}$ be the collection of finitely generated groups and $\mathcal{G}$ the collection of complex reductive affine algebraic groups.  Then, since the $G$-character variety of the trivial group $\Gamma=\{e\}$ is a point for any $G$, the set $\{\mathfrak{X}(\Gamma, G)\ |\ \Gamma\in\mathcal{F},\ G\in \mathcal{G}\}$ generates the Grothendieck ring of affine varieties.  By Bumagin and Wise \cite{BW} {\it all} countable groups arise as $\mathrm{Out}(\Gamma)$ as we vary finitely generated $\Gamma$.  So we have a very rich class of natural dynamical systems $(\mathfrak{X}(\Gamma, G),\mathrm{Out}(\Gamma))$; interesting only when both $\mathfrak{X}(\Gamma,G)$ and $\mathrm{Out}(\Gamma)$ are non-trivial.

\subsection*{Some History}

A classical example to consider is when $G=\mathrm{U}(1)$ and $\Gamma=\Gamma_g$ is the fundamental group of a closed orientable surface $S_g$ of genus $g$.  Then $\X(\Gamma_g, G)=\mathrm{Hom}(\Gamma_g, \mathrm{U}(1))\cong \mathrm{U}(1)^{2g}$ which is identified with the Jacobian of $S_g$, parametrizing topologically trivial holomorphic line bundles over a Riemannian surface diffeomorphic to $S_g$.  The action of $\mathrm{Out}(\Gamma_g)$ factors through $\mathrm{GL}(2g,\Z)$ acting on a product of circles, and is classically ergodic.  The Dehn-Nielsen-Baer Theorem \cite{primer} states that for such surfaces the (extended) mapping class group of isotopy classes of all homeomorphisms of $S_g$ is isomorphic to $\mathrm{Out}(\Gamma_g)$. 

The above example generalizes to arbitrary compact Lie groups $K$ as follows.  The smooth locus of the character variety $\X(\Gamma_g,K)$ is symplectic, as proven by Goldman \cite{G-Sym}, and the associated invariant measure is finite, as shown by Huebschmann \cite{Hueb}.  With respect to this measure $\mathrm{Out}(\Gamma_g)$ acts ergodically on $\X(\Gamma_g,K)$ when $g\geq 2$, see Goldman \cite{gol_erg} and Pickrell-Xia \cite{PX1}.  For  non-orientable surfaces $N$ with $\chi(N)\leq -2$ and $G=\SU(2)$  Palesi \cite{pal_erg} shows ergodicity with respect to a different invariant measure.  

The situation for non-compact Lie groups $G$ is even more complicated, and we need two definitions to describe the situation.  First, for a semisimple Lie group $G$ with no compact factors and a rank $r$ free group $F_r$, a homomorphism $\rho : F_r \to G$ is said to be {\it primitive stable} if any bi-infinite geodesic in the Cayley graph of $F_r$ defined by a primitive element (a member of free generating set) is mapped under the orbit map to a uniformly Morse quasigeodesic in the symmetric space $G/K$ where $K$ is a maximal compact in $G$. Second, if $\rho$ is injective, $\rho(F_r)$ is discrete, and $\rho(F_r)$ preserves and acts cocompactly on some nonempty convex subset of $G/K$ then $\rho$ is {\it convex cocompact}.

In the case of $G=\mathrm{SL}(2,\C)$, Minsky \cite{min_ond} showed the set $\X_{PS}(F_{n},G)$ of primitive stable representations is a domain of discontinuity for the action of $\mathrm{Out}(F_r)$ on $\X(F_{r},G)$.  Minsky also proved that $\X_{PS}(F_{n},G)$ strictly contains the set $\X_{CC}(F_{r},G)$ of convex cocompact representations that is also a domain of discontinuity for the action of $\mathrm{Out}(F_r)$ on $\X(F_{r},G)$. When $\Gamma = F_2$ is the fundamental group of the $1$--holed torus $S_{1,1}$, Bowditch \cite{bow_mar}, in the case of quasi-Fuchsian representations, and more generally Tan-Wong-Zhang \cite{tan_gen}, showed that there is a domain of discontinuity $\X_{BQ}(F_{2},G)$ for the action of the mapping class group of $S_{1,1,}$ on the corresponding character variety $\X(F_{2},G)$.  The set $\X_{BQ}(F_{2},G)$ is the set of {\it Bowditch representations} where the comparable conditions (for $G=\SL(2,\C)$ ) from Theorem \ref{Main-characterization} hold true  (although \cite{bow_mar, tan_gen} do not show $\X_{BQ}(F_{2},G)$ strictly contains the convex cocompact representations). Later, Lee-Xu \cite{lee_xu} and Series \cite{series2019,series2020} proved $\X_{PS}(F_{2},G)=\X_{BQ}(F_{2},G)$. Schlich \cite{schlich} generalized this result to cases where $G$ is the group of isometries of a $\delta$-hyperbolic space. We will discuss her results below in more detail, since they are connected with the results we present. 

Following the work of Bowditch, there have been other cases where interesting domains of discontinuity have been described for relative $\mathrm{SL}(2,\C)$-character varieties of surface groups. For example, you can see the description of a domain of discontinuity, the Bowditch set $\X_{BQ}(\pi_1(S), G)$, for the action of the mapping class group $\mathrm{Mod}(S)$ on the variety $\X(\pi_1(S), G)$, by Maloni-Palesi-Tan \cite{mal_ont} when $S$ is $4$--holed sphere and by Maloni-Palesi \cite{MaPa1} when $S$ is the $3$--holed projective plane. In the case when $\Gamma$ is the fundamental group of a hyperbolic $3$--manifold there has also been significant work finding interesting domains of discontinuity as described in this survey of Canary \cite{can_dyn}. When $G$ is a semisimple group of higher rank and without compact factors, Kim-Kim \cite{kim_kim_2020} generalized Minsky's work on primitive-stable representations and showed that they form an open domain of discontinuity for the action of $\mathrm{Out}(\Gamma)$ on $\X(\Gamma,G)$ when $\Gamma$ is a free group of rank at least $2$ or the fundamental group of a closed surface of genus $g\geq 2$.

Motivated by Lubotzky \cite{Lub}, Minsky \cite{min_ond} and Gelander-Minsky \cite{GeMi} discussed primitive-stable and redundant representations in the case of free groups and $G=\mathrm{SL}(2,\C)$. Inspired by their work, one can imagine a decomposition of $\X(\Gamma,G)$ into an open set where the action of $\mathrm{Out}(\Gamma)$ is properly discontinuous and a closed set where the action is chaotic. The only example fully understood where $G$ is {\it not} compact that exemplifies such a dynamical dichotomy is when $\Gamma$ is a free group of rank 2 and $G = \mathrm{SL}(2,\R)$, see Goldman \cite{gol_the}, Goldman-McShane-Stantchev-Tan \cite{gol_dyn} and March\'e-Wolff \cite{MarcheWolff16, MarcheWolff19}.

As summarized in Goldman \cite{Goldman-survey}, when $G$ is compact, (relative) character varieties generally exhibits nontrivial homotopy, and the action of the mapping class group exhibits nontrivial chaotic dynamics. In contrast, when $G$ is non-compact, then $\X(\Gamma_g,G)$ contains open sets (like Teichm\"uller space or more generally Hitchin components) which are contractible and admit properly discontinuous actions.  Philosophically, the action of $\mathrm{Out}(\Gamma)$ on (relative) character varieties exhibits properties of these two extremes.  

\subsection*{Relation with work of Schlich} 

Schlich \cite{schlich} considers representations from the fundamental group $\Gamma = \pi_1(S)$ of the $1$--holed torus $S=S_{1,1}$ into the isometry group $G$ of a geodesic and proper $\delta$--hyperbolic space. Since $\mathrm{P}\SUtwo =\mathrm{Isom}(\mathbf{H}^2_{\C})$ is the isometry group of the complex hyperbolic plane (see Section \ref{sec-ch2}), her result applies to the setting we study. For the character variety $\X = \X(F_2,\SU(2,1))$ Schlich defines two subsets of the character variety: the Bowditch set (using Definition (BQ3)) and the set of primitive-stable representations (generalizing Minsky's definition to this setting).  One of the main results of her thesis and paper shows that the two sets coincides. Since it is not hard to see that the set of primitive-stable representations is a domain of discontinuity for the action of $\Out(F_2)$ on $\X$, she also proves as a corollary of her result, that the set of Bowditch representations also has that property. In her thesis, she also considers the case of $S = S_{0,4}$ a four-holed sphere, and proves that in that setting the set of ``simple-stable'' representations coincides with the set of Bowditch's representations \cite{schlich-thesis}.

The main difference between our work and Schlich's work is in the way Bowditch representations are defined. This also differentiates what are the challenging aspects of the work. We define our Bowditch set explicitly in terms of the character variety using the symmetry and structure understood from \cite{law0,law1}.  Our approach is similar to that of \cite{bow_mar} and its generalizations \cite{tan_gen, mal_ont, MaPa1}, but adapted to representations into $\SU(2,1)$ where the character variety is significantly more complicated. In particular, our definition is well-adapted to computer implementations and testing since it is a definition dependent on only a finite number of conditions. Schlich's work generalizes  to isometry groups of Gromov hyperbolic spaces the result obtained for $\SL(2,\C)$ by Series \cite{series2019,series2020}and Lee-Xu \cite{lee_xu}  that for the free group $F_2$ the notions of Bowditch representations and primitive stable representations are equal. In fact, her result gives a new independent proof of this result for $\SL(2,\C)$.

\subsection*{Acknowledgments} 
The authors acknowledge support from U.S. National Science Foundation grants DMS 1107452, 1107263, 1107367 RNMS: ``Geometric Structures and Representation Varieties'' (the GEAR Network). Lawton was supported by the Simons Foundation. Maloni was partially supported by U.S. National Science Foundation grant DMS-1848346 (NSF CAREER). Lastly, we thank the referees for many helpful suggestions which improved the paper.

\section{Background}\label{s:not}

In this section we fix the notation which we will use in the rest of the paper and give some important definitions. 

\subsection{Combinatorics of simple closed curves on the one-holed torus \texorpdfstring{$S$}{S}}

Let $S= S_{1,1}$ be the one-holed torus. Its fundamental group $\Gamma = \pi_1(S)$ is isomorphic to the non-abelian rank $2$ free group $F_2:=\langle \alpha, \beta \rangle$. Abusing notation, we will let $\alpha$ and $\beta$ represent the (free homotopy class of) loops associated to the two generators so that the peripheral curve $\gamma$ going around the boundary component of $S$ is oriented to be homotopic to $[\alpha,\beta]:=\alpha\beta\alpha^{-1}\beta^{-1}$. 

Let $\Sc$ be the set of free homotopy classes of essential unoriented simple closed curves on $S$. Recall that a curve is \emph{essential} if it is non-trivial and non-peripheral, or equivalently if it does not bound a disc, or a one-holed disk. We will generally omit the word essential from now on. We can identify $\Sc$ to a well-defined subset of $\Gamma/\!\sim$, where the equivalence relation $\sim$ on $\Gamma$ is defined as follows: $g \sim h$ if and only if $g$ is conjugate to $h$ or $h^{-1}$. 

\subsection{Various simplicial complexes associated to \texorpdfstring{$S$}{S}}\label{ss:simple}

\begin{Definition}
The curve complex $\mathcal{C}$ of the 1-holed torus $S$ is defined by setting $k$--simplices to be subsets of $k+1$ distinct (free homotopy classes of) essential simple closed curves in $S$ that pairwise intersect once (up to homotopy).
\end{Definition}

With that definition $\mathcal{C}$ is a $2$--dimensional simplicial complex isomorphic to the Farey complex \cite{Minsky-geom}. 

\begin{figure}
[hbt] \centering
\includegraphics[height=6 cm]{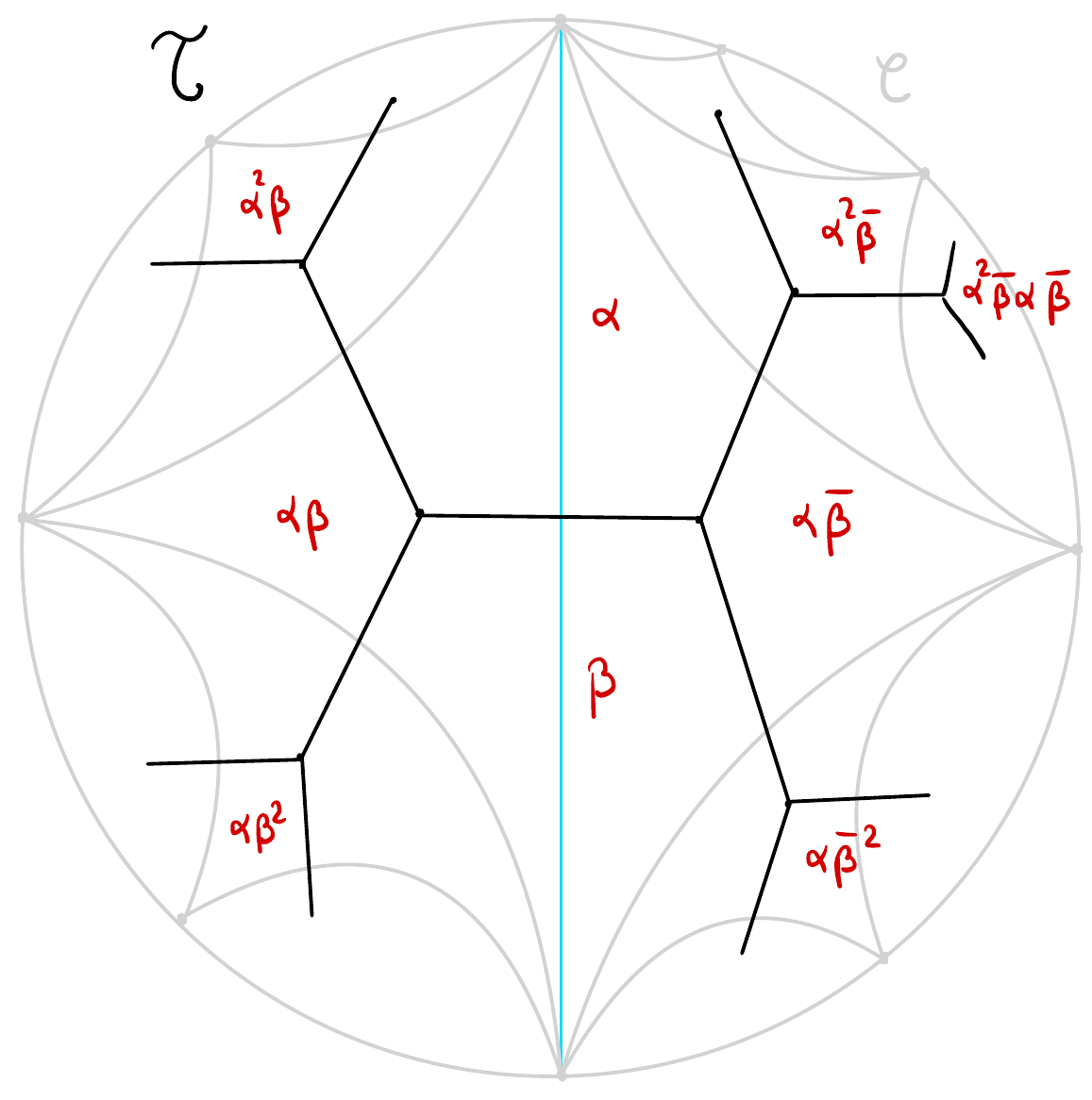}
\caption{The complexes $\mathcal{C}$ (in grey) and $\mathcal{T}$ (in black), and the labeling of the connected components of $\mathbf{H}^2 \setminus \mathcal{T}$ once we fix the generators of $F_2 = \langle \alpha, \beta \rangle$.}
\label{fig:f-t}
\end{figure}

\begin{Definition}
Let $\mathcal{T}$ be the simplicial complex defined by letting the set of $k$--simplices in $\mathcal{T}$, denoted by $\mathcal{T}^{(k)}$, be the set $\mathcal{T}^{(k)} = \mathcal{C}^{(2-k)}$.  The complex $\mathcal{T}$ is called the simplicial dual of $\mathcal{C}$ and is a countably infinite simplicial tree properly embedded in  hyperbolic $2$--space $\HH$ whose vertices all have degree $3$. \end{Definition}

Figure \ref{fig:f-t} illustrates $\mathcal{C}$ and $\mathcal{T}$. Note that $\mathcal{T}$ can also be identified with the graph of ideal triangulations of $S$, where edges corresponds to edge-flips. Since we will not need this identification in the rest of the paper, we will not discuss the details. 

The set $\Omega:=\pi_0(\mathbf{H}^2- \mathcal{T})$ is in one-to-one correspondence with the set of vertices of $\mathcal{C}$ and hence with the set $\Sc$; that is, $$\Omega \cong \mathcal{C}^{(0)} \cong \Sc.$$

Given a system of generators for $F_2$ corresponding to (free homotopy classes) of simple closed curves in $S$, we obtain a labeling of regions in $\Omega$.  For example, with $F_2 = \langle \alpha, \beta \rangle$ we obtain the labeling in Figure \ref{fig:f-t}. As you can see, the labeling is defined from the starting central edge $e = \{\a, \b\}$ corresponding to the intersection of the regions in $\Omega$ labeled $\alpha$ and $\beta$, and is defined, for a certain region $X$,  as the product of the two regions adjacent to $X$ but closer to the initial edge $e$. Note that when doing this product the region $\beta$ needs to be considered as $\beta$ on the ``left side" of the edge $e$ and $\beta^{-1}$ on the ``right side" of the edge $e$.

\begin{figure}[hbt] \centering
\includegraphics[height=7 cm]{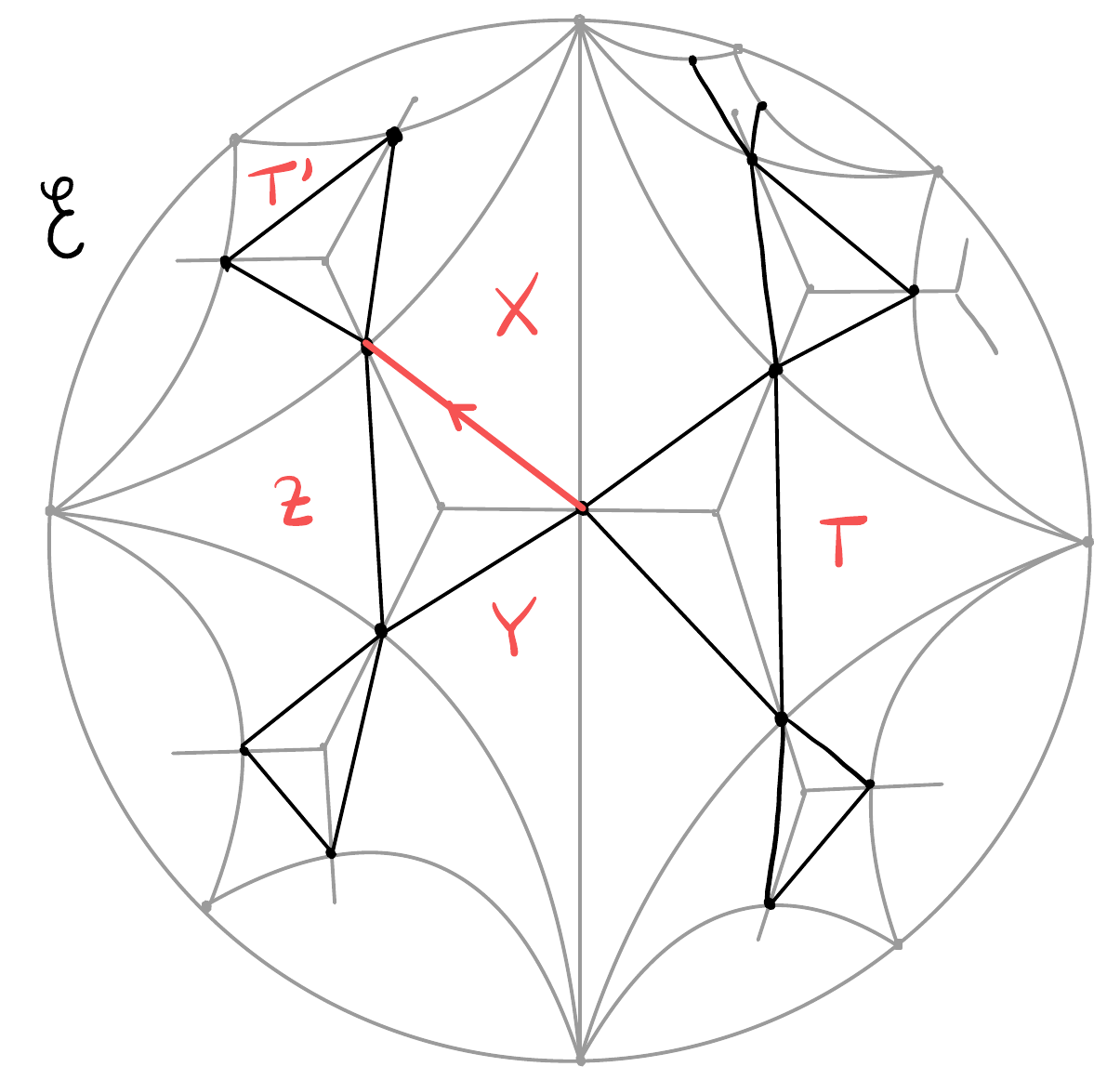}
\caption{Oriented edge $(X, Y, Z; T \to T')$. If you forget the orientation, the same red edge can be denoted $(X; Y, Z)$.}
\label{fig:e_oriented}
\end{figure}

\begin{Definition}
The \textit{edge graph} $\mathcal{E}$ is the graph whose vertices correspond to edges of $\mathcal{T}$, and whose edges correspond to adjacency. More precisely, if we denote $\mathcal{E}^{(k)}$ the set of $k$--simplices, we have that $\mathcal{E}^{(0)} = \mathcal{T}^{(1)} = \mathcal{C}^{(1)}$ and edges of $\mathcal{E}$ corresponds to edges that are the boundary of the same triangle in $\mathcal{C}^{(2)}$. 
\end{Definition}

Note that $\mathcal{E}$ is quasi-isometric to $\mathcal{T}$. In fact, one can see that $\mathcal{E}$ is a ``trivalent tree of triangles" in the sense that it is a graph made of triangles (one for each triangle $T$ in $\mathcal{C}^{(2)}$ or, equivalently, one for each vertex of $\mathcal T$).

\begin{figure}[hbt] \centering
\includegraphics[height=7 cm]{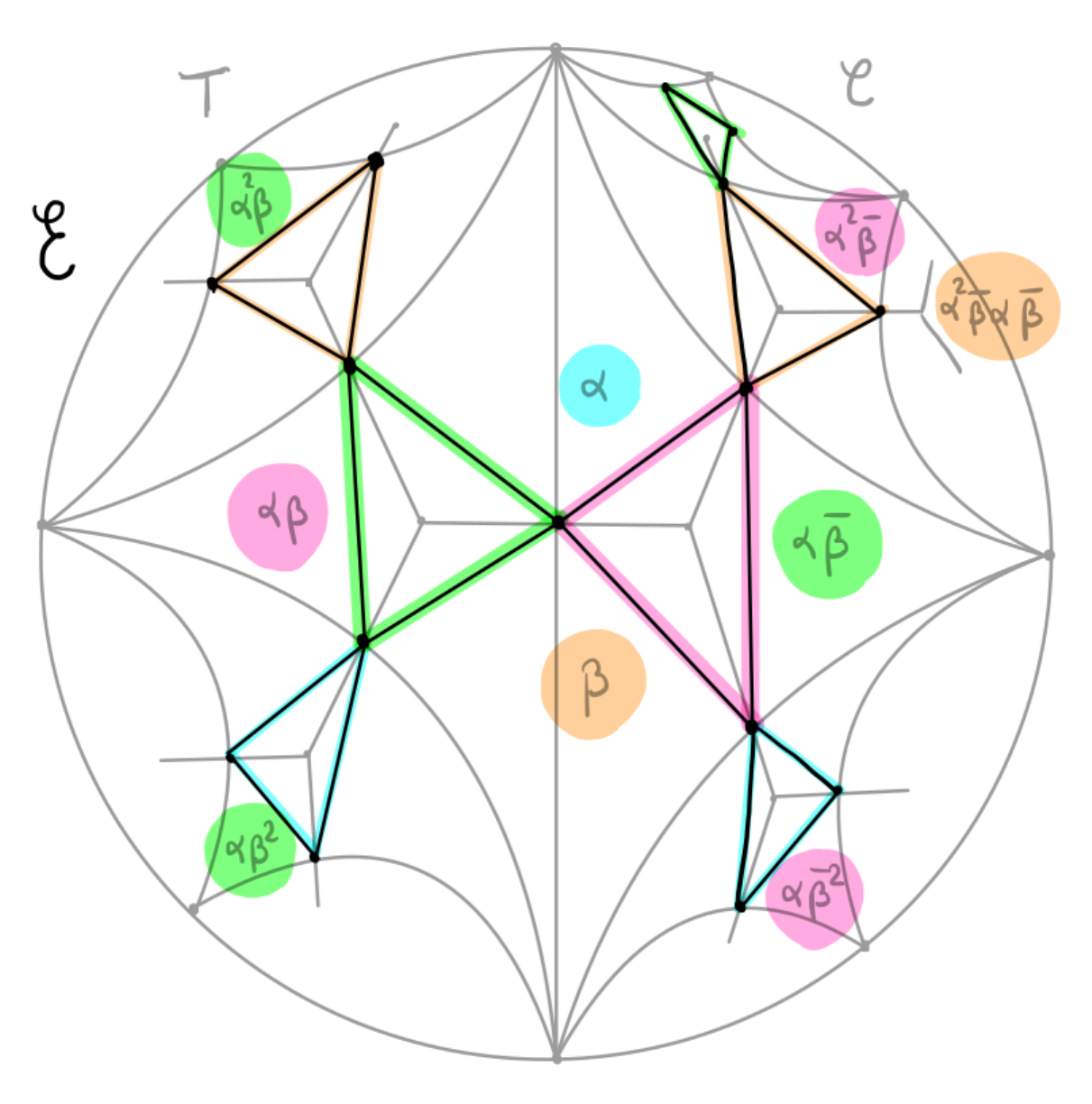}
\caption{The complex $\mathcal{E}$ with the coloring of its edges and regions of $\HH \setminus \mathcal{T}$.}
\label{fig:e}
\end{figure}

Since each edge of $\mathcal{T}$ is uniquely determined by two (free homotopy classes of) simple closed curves intersecting once, we will denote vertices of $\mathcal E$ as $v = (\alpha, \beta) \in \mathcal{E}^{(0)}$. It will be convenient for some proofs to denote such a vertex $v$ as $(\a,\b;\g,\d)$, where $\g,\d$ are regions adjacent to $(\a,\b)$ considered as an element in $\mathcal{T}^{(1)}$. Each edge of $\mathcal E$ is uniquely determined by the choice of a triangle $T$ in $\mathcal{C}^{(2)}$ plus the choice of one of its vertices, so we will denote edges of $\mathcal E$ as $e = (\alpha; \beta, \gamma) \in \mathcal{E}^{(1)}$. We  will also have to denote oriented edges, so in that case we will use the notation $(\alpha, \beta, \gamma; \delta, \delta')$. See Figure \ref{fig:e_oriented}.

Additionally, we will have to color the edges of $\mathcal{E}$ with four colors so that the edges of each triangle have the same color and each triangle is adjacent to triangles with distinct coloring. This coloring can be extended to color the vertices of $\mathcal{C}$ (or equivalently to the connected components of $\HH \setminus \mathcal{T}$).  This is possible since each such region is bounded by triangles colored with three distinct colors.  Therefore, the region will be colored in the color that does not appear in the triangles going around the region.  See Figure \ref{fig:e}.

\subsection{Complex Hyperbolic Geometry \texorpdfstring{$\HH_{\mathbb{C}}$}{H2C}}\label{sec-ch2} 

Let $\C^{n,1}\cong \C^n\times \C$ be the $(n+1)$-dimensional complex vector space with Hermitian pairing defined for vectors $(\mathbf{x},x_{n+1})\in \C^{n+1}$ by $$\langle (\mathbf{x},x_{n+1}),(\mathbf{y},y_{n+1})\rangle:=\left(\sum_{i=1}^{n}x_i\overline{y_i}\right)-x_{n+1}\overline{y_{n+1}}.$$  We say a vector $\mathbf{v}$ is {\it negative} if and only if $\langle \mathbf{v},\mathbf{v}\rangle<0$.  Let $\mathbf{P}(\C^{n,1}):=(\C^{n,1}-\{\mathbf{0}\})/\C^*$ be the space of lines in $\C^{n,1}$ through the origin. Then {\it complex hyperbolic space} of dimension $n$ is defined to be $\mathbf{H}_\C^n:=\{[\mathbf{v}]\in \mathbf{P}(\C^{n,1})\ |\ \langle \mathbf{v},\mathbf{v}\rangle<0\}$.

There is a natural copy of $\C^n$ in $\mathbf{P}(\C^{n,1})$ given by $$\mathbb{A}:=\{[(\mathbf{v},v_{n+1})]\in \mathbf{P}(\C^{n,1})\ |\ v_{n+1}\not=0\}.$$  The biholomorphism $\varphi:\C^n\to \mathbb{A}$ is given by $\varphi(\mathbf{v})=[(\mathbf{v},1)]$.  We observe that $\mathbf{H}_\C^n\subset \mathbb{A}$ since if $v_{n+1}=0$ then $\langle \mathbf{v},\mathbf{v}\rangle>0$.  Then restricting $\varphi^{-1}$ to $\mathbf{H}_\C^n$ we see that $\mathbf{H}_\C^n$ is biholomorphic to the unit ball in $\C^n$ since $\langle (\mathbf{v},1),(\mathbf{v},1)\rangle<0$ if and only if $\sum_{i=1}^{n}v_i\overline{v_i}<1$.

Consequently, the boundary of $\mathbf{H}_\C^n$ is homeomorphic to the unit sphere in $\C^n$: $$\partial\mathbf{H}_\C^n=\{[\mathbf{v}]\in \mathbf{P}(\C^{n,1})\ |\ \langle \mathbf{v},\mathbf{v}\rangle=0\}\cong S^{2n-1}.$$

We let $\mathrm{U}(n,1)$ be the subgroup of $\mathrm{GL}(n+1,\C)$ consisting of matrices $A$ such that  $$\langle A(\mathbf{x},x_{n+1}),A(\mathbf{y},y_{n+1})\rangle=\langle (\mathbf{x},x_{n+1}),(\mathbf{y},y_{n+1})\rangle.$$  The image of $\mathrm{U}(n,1)$ in $\mathrm{PGL}(n+1,\C)$, denoted $\mathrm{PU}(n,1)$, is the group of biholomorphisms of $\mathbf{H}_\C^n$ (see \cite{gol-chyp}).

$\mathrm{PU}(n,1)$ acts transitively on $\mathbf{H}_\C^n$ with stabilizer isomorphic to $\mathrm{U}(n)$.  Thus, we see that complex hyperbolic space is a homogeneous space $$\mathbf{H}_\C^n\cong \mathrm{PU}(n,1)/\mathrm{U}(n)\cong\mathrm{SU}(n,1)/\mathrm{S}(\mathrm{U}(n)\times \mathrm{U}(1)).$$

\subsection{\texorpdfstring{$\SLthreeC$}{SL3C} and \texorpdfstring{$\SUtwo$}{SU21}}

Let $\SLthreeC$ be the Lie group of $3\times 3$ complex matrices of unit determinant, which is also a complex affine variety since it is the zero locus of the determinant polynomial minus $1$.  

We review a few facts about two real forms of $\SLthreeC$.  First, $$\SU(3)=\{A\in \SLthreeC\ |\ A^{-1}=\overline{A}^\dagger\},$$ where $\dagger$ means ``transpose".  Letting $$J=\left(\begin{array}{ccc}1&0&0\\0&1&0\\0&0&-1\end{array}\right),$$ we secondly have $$\SU(2,1)=\{A\in\SLthreeC\ |\ A^{-1}=J\overline{A}^\dagger J\}.$$ Consequently, $\tr(A^{-1})=\overline{\tr(A)}$ for any $A\in \SU(3)\cup \SU(2,1)$.

The following lemma is well-known (see for example \cite[Exercise 3.1.1]{gol-chyp}), but we include a proof since it is an essential fact in our proof of Theorem \ref{Main2}. 

\begin{Lemma}\label{lem:maxcomp}
Any maximal compact subgroup of $\SU(2,1)$ is isomorphic to $\mathrm{U}(2)$.  In addition, the center of $\SU(2,1)$ is isomorphic to $\Z_3$.
\end{Lemma}

\begin{proof}
In general, the group $\mathrm{U}(p)\times\mathrm{U}(q)$ is a maximal compact in $\mathrm{U}(p,q)$ and so $\mathrm{S}(\mathrm{U}(p)\times\mathrm{U}(q))$ is a maximal compact in $\SU(p,q)$.  In the case of $\SU(2,1)$ we have that $\mathrm{S}(\mathrm{U}(2)\times\mathrm{U}(1))=\{(A,\lambda)\in \mathrm{U}(2)\times \mathrm{U}(1)\ |\ \mathrm{Det}(A)\lambda=1\}$.  On the other hand, this group is isomorphic to $\mathrm{U}(2)$ simply by sending $(A,\lambda)\mapsto A$.  Since all maximal compact subgroups are isomorphic, the first result follows.

Now we compute the center.  Suppose $Z\in \SU(2,1)$ is in the center $Z(\SU(2,1))\leq \SU(2,1)$.  Then for all $A\in \SU(2,1)$ we have $AZ=ZA$.  On the other hand this equation, which is necessarily algebraic, holds on the Zariski closure of $\SU(2,1)$, which is $\SL(3,\C)$.  So $Z$ is a central matrix in $\SL(3,\C)$ and so a scalar multiple of the identity.  The only such matrices are $\{\omega I, \omega^2 I, I\}\cong \Z_3$, where $\omega$ is a primitive cube root of unity.
\end{proof}

An automorphism of complex hyperbolic space $\mathbf{H}^2_\C\subset \mathbb{P}(\C^{2,1})$, represented by an element in $ \mathrm{PSU}(2,1):=\SU(2,1)/Z(\SU(2,1))$, lifts to an element of $\mathrm{SU}(2,1)$, and the fixed points in $\mathbf{H}^2_\C\cup\partial\mathbf{H}^2_\C$ correspond to eigenvectors of this lift.  For each such automorphism, there are three lifts corresponding to the cubic roots of unity $\zeta_3\subset S^1\subset \C$.  The following definition is taken from \cite[Section 6.2.1]{gol-chyp}.
\begin{Definition}
An element in $\mathrm{SU}(2,1)$ is called:
\begin{enumerate}
	\item {\it elliptic} if it has a fixed point in $\mathbf{H}^2_\C$; 
	\item {\it parabolic} if it has a unique fixed point on $\partial\mathbf{H}^2_\C$; 
	\item {\it unipotent} if it is parabolic and all the eigenvalues are equal;
	\item {\it loxodromic} if it fixes a unique pair of points on $\partial\mathbf{H}^2_\C$.
\end{enumerate}
\end{Definition}

Given a matrix $A\in\rm{SU} (2,1)$ such that $t = \tr(A)$, its characteristic polynomial, whose zeros are the eigenvalues of $A$, is the polynomial $\chi_A(x) = x^3 - t x^2 + \overline t x -1 $. The resultant of $\chi_A$ and its derivative tells us when $\chi_A$ has multiple roots, and hence when $A$ has an eigenvalue with multiplicity greater than $1$. Let $f\co\C \to \R$ be defined by
$$f(t) = \mathrm{Res} (\chi_A , \chi_A') =  |t|^4 - 8 \mathrm{Re} (t^3) + 18 |t|^2 - 27.$$
This function vanishes on a deltoid centered at $0$ of radius $3$, see Figure \ref{fig:deltoid}.

\begin{figure}
\includegraphics[height=6cm]{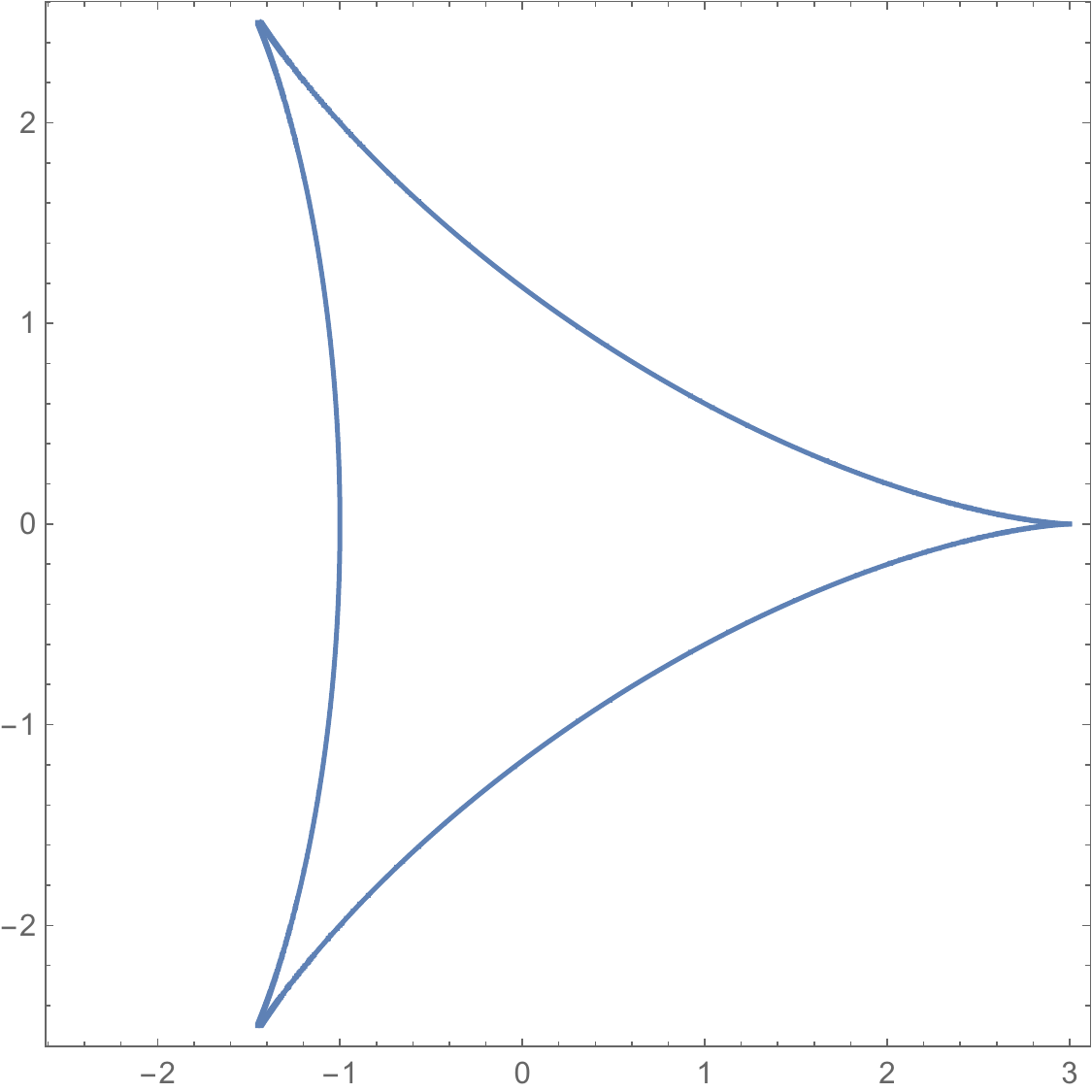}
\caption{$f(t)=0$}
\label{fig:deltoid}
\end{figure}

\begin{Theorem}[Theorem 6.2.4, \cite{gol-chyp}]
The map $\tr\co\mathrm{SU}(2,1)\to\C$ is surjective.  If $A,B\in \mathrm{SU}(2,1)$ are loxodromic and $\tr(A)=\tr(B)$, then $A$ and $B$ are conjugate.  

In addition, for any $A\in\mathrm{SU}(2,1)$ with $t = \tr(A)$ we have:
\begin{enumerate}                                                                                                                                             \item[$(1)$] $A$ is elliptic with distinct eigenvalues if and only if $f(t)<0$;
\item[$(2)$] $A$ is loxodromic if and only if $f(t)>0$;
\item[$(3)$] $A$ is parabolic but not unipotent if and only if $A$ is not elliptic and $t\in f^{-1}(0)-3\zeta_3$; 
\item[$(4)$] $A$ is a complex-reflection if and only if $A$ is elliptic and $t\in f^{-1}(0)-3\zeta_3$;
\item[$(5)$] $A$ is an unipotent automorphism of $\mathbf{H}^2_\C$ or the identity map if and only if $t\in 3\zeta_3$.
\end{enumerate}
\end{Theorem}

The following result is a corollary to the above theorem.
\begin{Corollary}[\cite{gol-chyp}]\label{corol_f}
	Let $A\in\mathrm{SU}(2,1)$ with eigenvalues $\lambda_1,\lambda_2,\lambda_3$, and let $t = \tr(A)$. We have:
\begin{enumerate}
\item[$(1)$] If $f(t) > 0$, then $A$ is loxodromic, and its eigenvalues are all distinct: one is inside, one is outside, and one is on the unit circle.
\item[$(2)$] If $f(t) = 0$, then $A$ has an eigenvalue of multiplicity at least $2$, and all eigenvalues are on the unit circle.
\item[$(3)$] If $f(t) < 0$, then $A$ is regular elliptic, and has three distinct eigenvalues satisfying $|\lambda_i | = 1$.
\end{enumerate}
\end{Corollary}

\subsubsection{\texorpdfstring{$\SLthreeC$}{SL3C}-Character Variety}\label{sl3-charvar}

Let $S_{1,1}$ be the one-holed torus.  The fundamental group of $S_{1,1}$ is isomorphic to a rank $2$ free group $$F_2:=\langle \alpha, \beta\rangle=\langle \alpha, \beta, \gamma\ |\ [\alpha, \beta]=\gamma\rangle,$$ where $\gamma$ corresponds to a loop around the boundary circle and $\alpha,\beta$ are loops of longitude and latitude in the torus, as we discussed in Section \ref{ss:simple}.

The set of group homomorphisms $\hm(F_2,\SLthreeC)$ is in bijective correspondence with $\SLthreeC\times\SLthreeC$ by sending a homomorphism $\rho$ to $(\rho(\alpha),\rho(\beta))$.  Since products of smooth manifolds are smooth manifolds and products of affine varieties are affine varieties, we can give the structure of a complex affine algebraic variety (and a smooth manifold) to $\hm(F_2,\SLthreeC)$.  

The group $\SLthreeC$ acts on $\hm(F_2,\SLthreeC)$ by conjugation which becomes {\it simultaneous} conjugation under the identification with  $\SLthreeC\times\SLthreeC$.  Precisely, for $g \in \SLthreeC$ and $\rho\in\hm(F_2,\SLthreeC)$ the action is defined by $g\cdot \rho=g\rho g^{-1}$, and for $(A,B)\in \SLthreeC\times\SLthreeC$ we have $g\cdot (A,B)=(gAg^{-1}, gBg^{-1})$.  

We are interested in the moduli space of such homomorphisms up to the equivalence defined by this action.  The quotient $\hm(F_2,\SLthreeC)/\SLthreeC$ does not satisfy the $T_1$-separation axiom since some points are not closed (corresponding to non-closed orbits in $\hm(F_2,\SLthreeC)$).  There is a natural way to make it satisfy the $T_1$-separation axiom: remove the orbits that are not closed in $\hm(F_2,\SLthreeC)$.  Let $\hm(F_2,\SLthreeC)^*$ be the subspace of $\hm(F_2,\SLthreeC)$ that consists of homomorphisms with closed conjugation orbits. Then the {\it polystable quotient} $$\mathfrak{X}(F_2,\SLthreeC):=\hm(F_2,\SLthreeC)^*/\SLthreeC$$ satisfies the $T_1$-separation axiom. 

However, one can show that it is homeomorphic to the Geometric Invariant Theory (GIT) quotient $\hm(F_2,\SLthreeC)/\!\!/\SLthreeC$, where the latter is given the induced metric topology from being an affine variety (called the {\it analytic} or {\it strong} topology).  See Florentino-Lawton \cite[Theorem 2.1]{FlLa4} for a proof.  Thus, $\mathfrak{X}(F_2,\SLthreeC)$ has the structure of a stratified manifold (and so is a metric space).  Moreover, $\mathfrak{X}(F_2,\SLthreeC)$ is homotopic to the non-Hausdorff space $\hm(F_2,\SLthreeC)/\SLthreeC$ just the same (see Florentino-Lawton-Ramras \cite[Proposition 3.4]{FLR}). 

To understand the affine variety structure on $\mathfrak{X}(F_2,\SLthreeC)$, which approximates the non-Hausdorff space  $\hm(F_2,\SLthreeC)/\SLthreeC$, we need to understand the ring of polynomial functions on it. 

Let $\mathbf{X}=(x_{ij})_{1\leq i,j\leq 3}$ and $\mathbf{Y}=(y_{ij})_{1\leq i,j\leq 3}$ be two $3\times 3$ matrices of {\it variables} (called ``generic matrices").  Then the ring of polynomials on $\hm(F_2,\SLthreeC)$ is: $$\C[\hm(F_2,\SLthreeC)]\cong \C[x_{ij},y_{ij}|1\leq i,j\leq 3]/\langle \mathrm{Det}(\mathbf{X})-1,\mathrm{Det}(\mathbf{Y})-1\rangle.$$

Since $\SLthreeC$ acts algebraically by conjugation on $\hm(F_2,\SLthreeC)$, there is an action on $\C[\hm(F_2,\SLthreeC)]$ given by $(g\cdot f)(\rho)=f(g^{-1}\rho g)$ for $g\in \SLthreeC$. The ring of polynomials on $\mathfrak{X}(F_2,\SLthreeC)$ is then 
\begin{align*}
	\C[\X(F_2,\SLthreeC)]&:=\C[\hm(F_2,\SLthreeC)]^{\SLthreeC}\\
	& := \{f\in\C[\hm(F_2,\SLthreeC)]\ |\ g\cdot f=f\text{ for all } g\in \SLthreeC\}.
\end{align*}

One of the main theorems in Lawton \cite{law0,law1} states that $$\C[\X(F_2,\SLthreeC)]\cong \C[x,y,z,t,A,B,C,D,s]/\langle s^2-Ps+Q\rangle,$$ where $P,Q\in \C[x,y,z,t,A,B,C,D]$ and the polynomial $s^2-Ps+Q$ is irreducible.  

Consequently, we have:

\begin{Theorem}[\cite{law0,law1}]
$\X(F_2,\SLthreeC)$ is a hypersurface in $\C^9$ that is a branched double cover of $\C^8$.  The branching locus is exactly the representations that are fixed by transpose.
\end{Theorem}

The explicit form of $P$ and $Q$ are:
$$P=-3 + t D + x A - t A y - D x B + y B + x A y B - A B z - x y C + z C$$
and
$Q=9 + t^3 - 6 t D + D^3 + x^3 - 6 x A + t D x A + A^3 + t x^2 y + 
 3 t A y - 2 D^2 A y - t x A^2 y + D x y^2 + D A^2 y^2 + y^3 - 
 x A y^3 - 2 t^2 x B + 3 D x B - D x^2 A B + D A^2 B - 6 y B + 
 t D y B - x^3 y B + x A y B + x^2 A^2 y B - A^3 y B - t A y^2 B + 
 t x^2 B^2 + t A B^2 - D x y B^2 + x A y^2 B^2 + B^3 - x A B^3 - 
 3 t x z + D^2 x z + t A^2 z + t^2 y z - 3 D y z - D x A y z + 
 x^2 y^2 z + A y^2 z + x^2 B z + 3 A B z - x A^2 B z - t x y B z + 
 D B^2 z - A y B^2 z + D A z^2 - 2 x y z^2 + t B z^2 + z^3 + D x^2 C +
  t^2 A C - 3 D A C + 3 x y C - x^2 A y C + A^2 y C + t y^2 C - 
 3 t B C + D^2 B C - t x A B C - D A y B C - x y^2 B C + x B^2 C + 
 A^2 B^2 C - 6 z C + t D z C + x A z C + y B z C + t x C^2 + 
D y C^2 - 2 A B C^2 + C^3.$

As functions of the matrix variables $\mathbf{X}$ and $\mathbf{Y}$ (up to  the ideal $\langle \mathrm{Det}(\mathbf{X})-1,\mathrm{Det}(\mathbf{Y})-1\rangle$), the variables $x,y,z,t, A,B,C,D,s$ are as follows:

$$x=\tr(\mathbf{X}), \;\;y=\tr(\mathbf{Y}), \;\; z=\tr(\mathbf{X}\mathbf{Y}), \;\; t=\tr(\mathbf{X}\mathbf{Y}^{-1}),$$
$$A=\tr(\mathbf{X}^{-1}), \;B=\tr(\mathbf{Y}^{-1}), \;C=\tr(\mathbf{X}^{-1}\mathbf{Y}^{-1}), \;D=\tr(\mathbf{X}^{-1}\mathbf{Y}), \text{ and } s=\tr(\mathbf{X}\mathbf{Y}\mathbf{X}^{-1}\mathbf{Y}^{-1}).$$
In these terms, the relations are $$P=\tr(\mathbf{X}\mathbf{Y}\mathbf{X}^{-1}\mathbf{Y}^{-1})+\tr(\mathbf{Y}\mathbf{X}\mathbf{Y}^{-1}\mathbf{X}^{-1})$$ and $$Q=\tr(\mathbf{X}\mathbf{Y}\mathbf{X}^{-1}\mathbf{Y}^{-1})\tr(\mathbf{Y}\mathbf{X}\mathbf{Y}^{-1}\mathbf{X}^{-1}).$$

\begin{Remark}
In Lawton \cite{law0, law1} the polynomials $P,Q$ are shown to have $\mathrm{Out}(F_2)$ symmetry and in terms of that symmetry they can be written more compactly.  See also Parker \cite{Parker} for a ``cyclically" symmetric coordinate change.
\end{Remark}

As shown in  Florentino-Lawton \cite{FlLa1}, the singular variety $\X(F_2,\SLthreeC)$ strong deformation retracts to a topological $8$-sphere $S^8$. In fact, this $8$-sphere is precisely the $\mathrm{SU}(3)$-character variety of $F_2$:  $\X(F_2,\SU(3)):=\hm(F_2,\SU(3))/\SU(3).$

\subsubsection{\texorpdfstring{$\mathrm{SU}(2,1)$}{SU21}-Character Variety}

We are interested in the corresponding moduli space where $\SU(3)$ is replaced by $\mathrm{SU}(2,1)$.

Given a word $w\in F_2=\langle \alpha,\beta\rangle$, define $\tr_w\!:\!\mathfrak{X}(F_2,\SLthreeC)\to \C$ by $\tr_w([\rho]):=\tr(\rho(w))$. For representations taking values in $\SU(3)$ or $\SU(2,1)$, we have that $\tr_{w^{-1}}=\overline{\tr_w}$ for any word $w\in F_2$.  Thus, the real coordinate ring of $\X(F_2,\SU(3))$ or the polystable quotient $\X(F_2,\SU(2,1)):=\hm(F_2,\SU(2,1))^*/\SU(2,1)$ is generated by the real and imaginary parts of $$\{\tr_\alpha,\tr_\beta,\tr_{\alpha\beta},\tr_{\alpha\beta^{-1}},\tr_{\alpha\beta\alpha^{-1}\beta^{-1}}\},$$ subject to the relations \begin{equation}\label{branch-eq1}\mathrm{Re}(\tr_{\alpha\beta\alpha^{-1}\beta^{-1}})=P/2\end{equation} and \begin{equation}\label{branch-eq2}\mathrm{Im}(\tr_{\alpha\beta\alpha^{-1}\beta^{-1}})=\pm\sqrt{Q-P^2/4}.\end{equation}  Note that both $P$ and $Q$ will be real polynomials, as $P$ is the sum $$P=\tr_{\alpha\beta\alpha^{-1}\beta^{-1}}+\tr_{\beta\alpha\beta^{-1}\alpha^{-1}}=2\mathrm{Re}(\tr_{\alpha\beta\alpha^{-1}\beta^{-1}}),$$ while $Q$ is the product $$Q=\tr_{\alpha\beta\alpha^{-1}\beta^{-1}}\cdot\tr_{\beta\alpha\beta^{-1}\alpha^{-1}}=|\tr_{\alpha\beta\alpha^{-1}\beta^{-1}}|^2,$$ and hence will simplify to {\it real} polynomials in the real and imaginary parts of $\tr_\alpha,$ $\tr_\beta,$ $\tr_{\alpha\beta},$ and $\tr_{\alpha\beta^{-1}}$.

If we write this out in terms of the variables $x,y,z,t$ above we obtain for $P$ and $Q$ the following expressions:
$$P = |x|^2 |y|^2 + |x|^2 + |y|^2 + |z|^2 + |t|^2 - 2 \Re (xy\overline z ) - 2 \Re (\overline x y t ) - 3 ,$$ 
and
$Q=9 - 6 |x|^2- 6 | y |^2 - 6 |z|^2 - 6 | t |^2 +| t |^2 | y |^2 +| x |^2 | y |^2+|y|^2 |z|^2 + |t|^2|z|^2 +|x|^2|z|^2+|t|^2|x|^2 + |x|^2 |y|^4  + | y |^2 | x |^4  + 2 \mathsf{Re} (z^3)  + 2 \mathsf{Re} (t^3) + 2 \mathsf{Re}(x^3)  + 2 \mathsf{Re} (y^3)+ 2\mathsf{Re} (t x^2 y) - 6 \mathsf{Re} (t x z) + 2 \mathsf{Re} (t^2 y z) + 2 \mathsf{Re} (x^2 y^2 z) - 4 \mathsf{Re} (x y z^2)  +2 \mathsf{Re} (x y^2 \overline{t}) - 6 \mathsf{Re} (y z \overline{t}) + 2 \mathsf{Re} (x z \overline{t}^2)  +6 \mathsf{Re} (t y \overline{x}) - 2 | x |^2 \mathsf{Re} ( y^3) + 2 \mathsf{Re} (y^2 z  \overline{x}) -   2 | y |^2 \mathsf{Re} (\overline{x} t y)  -2 | x |^2 \mathsf{Re} ( y z  \overline{t}) + 2 \mathsf{Re} (z^2 \overline{t}\overline{x}) - 4 \mathsf{Re} ( y \overline{t}^2  \overline{x}) - 2 | x |^2 \mathsf{Re} (t y  \overline{x}) + 2 \mathsf{Re} (t z  \overline{x}^2) + 2 \mathsf{Re} (y^2 \overline{t} \overline{x}^2)  -2 | y |^2 \mathsf{Re} (x^3) + 2 \mathsf{Re} (x^2 z \overline{y}) - 2 | y |^2 \mathsf{Re} (t x  z) + 2 \mathsf{Re} (t z^2 \overline{y}) +6 \mathsf{Re} ( z  \overline{x} \overline{y}) - 2 | x |^2 \mathsf{Re} (z  \overline{x} \overline{y}) + 2 \mathsf{Re} (z \overline{t}\overline{y}^2) - 2 | y |^2 \mathsf{Re} ( z  \overline{x} \overline{y}).$

However, just because $\X(F_2,\SU(3))$ and $\X(F_2,\SU(2,1))$ satisfy the same real algebraic equations and have the same (real) dimension, it does not mean they are homeomorphic. In fact, $\X(F_2,\SU(3))\cong S^8$, and we now show that $\X(F_2,\SU(2,1))$ is homotopic to a torus $S^1\times S^1$.

\begin{Theorem}
 $\X(F_2,\SU(2,1))$ strong deformation retracts to $S^1\times S^1$.
\end{Theorem}

\begin{proof}
As defined above, $\X(F_2,\SU(2,1))$ is the polystable quotient of $\hm(F_2,\SU(2,1))$ by the conjugation action of $\SU(2,1)$.  As Lemma \ref{lem:maxcomp} shows, $\mathrm{U}(2)$ is a maximal compact subgroup of $\SU(2,1)$.  Thus, by Casimiro-Florentino-Lawton-Oliveira \cite{CFLO}, $\X(F_2,\SU(2,1))$ strong deformation retracts to $\X(F_2,\mathrm{U}(2))$.  

On the other hand, $\mathrm{U}(2)\cong \SU(2)\times_{\mathbb{Z}_2}S^1$, where $\SU(2)\times_{\mathbb{Z}_2}S^1 := \left(\SU(2)\times S^1\right)/ \mathbb{Z}_2,$ where $\mathbb{Z}_2$ acts on $\SU(2)\times S^1$ as follows. For $\mu \in \{\pm 1\}\cong\mathbb{Z}_2$, $A \in \SU(2)$ and $\zeta \in S^1$, $\mu \cdot (A, \zeta):= (\mu A, \mu^{-1}\zeta)$. From this we conclude that $\X(F_2,\mathrm{U}(2))\cong \X(F_2,\SU(2))\times_{\mathbb{Z}_2^2}(S^1\times S^1)$ by Florentino-Lawton \cite{FlLa2}, where $$\X(F_2,\SU(2))\times_{\mathbb{Z}_2^2}(S^1\times S^1) := \left( \X(F_2,\SU(2))\times (S^1\times S^1)\right)/ \mathbb{Z}_2^2$$ and the action of $\mathbb{Z}_2^2$ on $\X(F_2,\SU(2))\times (S^1\times S^1)$ is the product action coming from the action of $\mathbb{Z}_2$ defined above.  On the other hand, $\X(F_2,\SU(2))$ is a closed 3-ball $B$ (see Florentino-Lawton \cite{FlLa1}) and so by \cite{FlLa2} we conclude that $\X(F_2,\mathrm{U}(2))\cong B\times (S^1\times S^1),$ which strong deformation retracts to $S^1\times S^1$, as required.
\end{proof}

Given that $\X(F_2,\SLthreeC)$ deformation retracts  onto $\X(F_2,\SU(3))$ and $\X(F_2,\SU(3))$ is homeomorphic to an 8-sphere, we have the following corollary.

\begin{Corollary}
$\X(F_2,\SU(2,1))$ is not simply connected, and so is not homotopic to either $\X(F_2,\SLthreeC)$ or $\X(F_2,\SU(3))$.
\end{Corollary}

We will be studying the mapping class group dynamics on certain subvarieties of $\X(F_2,\SU(2,1))$, called {\it relative character varieties}.  To shorten the notation, let $\X=\X(F_2,\SU(2,1))$ in the sequel.  Define $\mathfrak{b}:\X\to \C$ by $\mathfrak{b}([\rho])=\tr(\rho(\alpha\beta\alpha^{-1}\beta^{-1}))$.  By \cite[Theorem 6.2.4]{gol-chyp}, $\mathfrak{b}$ is surjective.

\begin{Definition}
The {\it relative character variety} with respect to $c\in \C$ is $\X_c:=\mathfrak{b}^{-1}(c)$.
\end{Definition}

\begin{Theorem}
$c\in \C-\R$ if and only if $P^2-4Q\not=0$.  $\X_c$ is a smooth manifold if $c\in \C-\R$.
\end{Theorem}

\begin{proof}
The variety $\X(F_2,\SL(3,\C))$ is a branched double cover of $\C^8$ embedded in $\C^9$ as described in Section \ref{sl3-charvar}.  The branching locus $\mathcal{B}$, where the branched double cover is exactly 1-1, is defined by the equation $P^2-4Q=0$.  As shown in Florentino-Lawton \cite{FlLa1}, the locus $\mathcal{B}$ corresponds exactly to the representations that are fixed under the involution $[(A,B)]\mapsto [(A^\dagger,B^\dagger)]$. We call these the symmetric representations (up to conjugation).  Equations \eqref{branch-eq1} and \eqref{branch-eq2} imply that $\X_c\subset\mathcal{B}$ if and only if $c\in \R$.  Equivalently, $\X_c\cap\mathcal{B}=\emptyset$ if and only if $c\in \C-\R$.  As shown in Lawton \cite{law0,law1}, the reducible  locus (that is, the set of reducible representations) is contained in the branching locus.  Thus, if $c\in \C-\R$, we know that $\X_c$ consists of conjugation classes of irreducible representations.  The irreducible locus is a smooth manifold in $\X(F_2,\SL(3,\C))$ by \cite{law1}, which implies the irreducible locus in $\X$ is a smooth manifold too. Thus, since the boundary map is a submersion on the irreducible locus (as the computation in Goldman-Lawton-Xia \cite[Proposition 8.1]{GLX} applies to this setting too), the fibres $\X_c$ are smooth submanifolds by the inverse function theorem.  
\end{proof}

\begin{Remark}
The locus of reducible representations (which are generally singular) is a 5 dimensional subspace of the 7 dimensional branching locus in $\X$.  Thus, if $c\in\R$, and so $\X_c$ is in the branching locus, we expect that $\X_c$ contains singular representations and so is not a manifold. Morally this should follow since either $\X_c$ contains both an irreducible and a reducible and thus a singularity, or $\X_c$ contains only reducibles and thus contains an orbifold singularity from the action of the Weyl group.  However, as \cite[Proposition 3.1]{FGLM} shows, relative character varieties containing only reducible representations can ``accidentally" be manifolds.  We also note that irreducible representations in the branching locus may have special geometric features, as described by Paupert-Will \cite{paupert-will} or Will \cite{Will-07}.
\end{Remark}

\begin{Remark}
It would be interesting to explore whether $\X(F_2,\SU(2,1))$ is a topological manifold, since $\X(F_2,\SU(3))$ is a topological manifold and $\X(F_2,\SU(2,1))$ is homotopic to a manifold.  It would also be interesting to explore the topology $($up to homotopy$)$ of the level sets $\X_c$ for various values of $c\in \C$.
\end{Remark}

In our situation, $\mathrm{Out}(F_2)\cong \mathrm{GL}(2,\mathbb{Z})$ which is isomorphic to the mapping class group of the 1-holed torus. This group acts on $\X$ and the subgroup of mapping classes that fix the boundary of $S_{1,1}$ acts on $\X_c$. We will now proceed to show there exists an open domain of discontinuity $\X_{BQ}$ for the action of $\mathrm{Out}(F_2)$ on $\X$ which strictly contains the set of discrete, faithful, convex cocompact representations (more generally the primitive stable representations).

\section{Bowditch set} \label{sec:bow}

Recall that the aim of this article is to define a set of representations that will provide an open domain of discontinuity for the action of $\mathrm{Out}(F_2)$ on $\mathfrak{X}(F_2 , \SUtwo)$. In particular, in this section we will define and study the set of Bowditch representations, and in the next section we will prove that they define an open domain of discontinuity. 

For notational convenience, given a representation $\rho \co F_2 \rightarrow \SUtwo$, or its conjugacy class $[\rho] \in \X$, and given $X, Y \in \mathcal{S}$, we will write $x=\Tr(\rho(X))$, $y=\Tr(\rho(Y))$ et cetera. 

\subsection{Orientation on \texorpdfstring{$\mathcal{E}$}{E}}\label{orientation}

Given a representation $\rho \co F_2 \rightarrow \SUtwo$, we will need to assign to each undirected edge $e \in \mathcal{E}$ a particular directed edge $\vec e_\rho$. In Section \ref{ss:simple} we identified edges $e$ of $\mathcal{E}$ with triples $e = (X; Y, Z) \in \mathcal{E}^{(1)}$ where $X,Y,Z \in \mathcal{S}$. Since we need to discuss an orientation, it will be more useful to use the notation $e = (X, Y, Z; T , T') \in \mathcal{E}^{(1)}$; see Figure \ref{fig:e_oriented}.

\begin{Definition}
	Let $e= (X, Y, Z; T, T') \in \mathcal{E}^{(1)}$. Let  $t = |\tr (\rho (T)) | $ and $|\tr (\rho (T'))| = t' $, then we define the oriented edge $\vec e_\rho$ in the following way :
	\begin{itemize}
		\item If $|t| > |t'|$  (respectively $|t| < |t'|$) we orient the edge $e$ pointing towards $T'$ (respectively towards $T$).  
		\item If $|t| = |t'|$ we choose the orientation of $\vec e_\rho$ arbitrarily.
	\end{itemize}
	That oriented edge is denoted $\vec e_\rho  = (X, Y, Z; T \rightarrow T')$, when it points towards $T'$. 
\end{Definition}

We will see that the intersection of the Bowditch representations (up to conjugation) with the set of representations that satisfy $|t| = |t'|$ for infinitely many edges is empty. So for a given representation, the finite number of arbitrary choices that occur when $|t| = |t'|$ will not affect the large scale behavior of the orientation of the graph $\mathcal{E}$, the latter being what is essential to our proof.

\subsection{Connectedness and Fork Lemma} 

In this subsection we will prove an important lemma, called the ``Fork Lemma," which will be the key to prove the connectedness of the set of regions with trace bounded by a given constant.
	
\subsubsection{Fork Lemma} \label{sec:fork_lemma}

Let us first describe a preliminary result which will be useful in several proofs in this section and in Section \ref{sec:fibonacci}.	 

\begin{Lemma}\label{bound}
	Let $c \in \C$, $[\rho] \in \X_c$, $C=\max\{2\sqrt[4]{3},\sqrt{|\Re(c)|}\}$, and $v=(X,Y;Z,T) \in \mathcal{E}^{(0)}$. If the collection $x$, $y$, $z$, and $t$ satisfies $|x|,|y| \geq C$ and $|t| \geq |z|$, then we have:
$$\frac C4 \max (|x| , |y|) < |t| < 5 |x||y|.$$
\end{Lemma}

\begin{proof}
	Without loss of generality we can assume that $|x| \geq |y|\geq C$. For the first inequality, assume by contradiction that $|t| \leq \frac{C}{4}|x|$. Recall that : 
	$$P=2 \Re (c) = |x|^2 |y|^2 + |x|^2 + |y|^2 + |z|^2 + |t|^2 - 2 \Re (xy\overline z ) - 2 \Re (x \overline y t ) - 3.$$
	We have:
	\begin{align*}
		P &\geq (|x|^2|y|^2 - 4|x||y||t| +|t|^2) + |x|^2 + |y|^2 + |z|^2 -3 \\
		& \geq \left((|xy| - 2 |t|)^2 - 3|t|^2\right) + 2C^2 +0 - 3 \\
		& \geq \left(\left(|x||y| - \frac{C}{2} |x|\right)^2 - \frac{3}{16} C^2 |x|^2\right) + 2C^2 - 3 \\
		& = |x|^2 \left( \left(|y|-\frac{C}{2}\right)^2 - \frac{3}{16}C^2\right) + 2C^2 - 3 \\
		& \geq C^2 \frac{C^2}{16} + 2C^2 - 3 > 2C^2,
	\end{align*}
	where we used the hypothesis $|t| \leq \frac{C}{4}|x|$ only in the third line, and the fact that $C^4 > 16 \cdot 3$ in the last inequality. This gives a contradiction, as $ 2\Re(c) \leq 2 C^2$, and hence proves the first inequality.
	
	For the second inequality, assume by contradiction that $|t| \geq 5 |x| |y|$ (or equivalently $|x| |y| \leq \frac{|t|}{5}$). In that case, we have :
	\begin{align*}
		P& \geq (|x|^2|y|^2)+(|x|^2+|y|^2)+(|z|^2)+(|t|^2-2 |x||y||t| - 2 |x||y||z|)  - 3 \\
		& \geq  C^4 + 2C^2 + 0 + \frac{|t|^2}{5} - 3 \\
		& > 2C^2,
	\end{align*}
	where the last inequality is justified since $C^4 > 16\cdot 3 > 3$. This gives the same contradiction, as $P \leq 2 C^2$.
\end{proof}

We are now ready to state and prove the Fork Lemma. First, let us define what is a fork. For a vertex $v = (X , Y ; Z, T) \in \mathcal{E}^0$ choose representatives $\alpha, \beta \in F_2$ of $X$ and $Y$ such that $\alpha \beta$ and $\alpha \beta^{-1}$ are representatives of $Z$ and $T$, respectively. Recall that the traces of the four neighboring regions around $v$ are given by:
$$\tr (\alpha \beta^2) = t - x \overline{y} + y z,  \quad\quad\quad\tr \left((\alpha^2 \beta)^{-1}\right) = t - x \overline{y} + \overline{xz},$$ and $$ \tr (\a^{-2} \b)=\overline{tx} - x y + z,  \quad\quad\quad \tr (\a \b^{-2}) = -x y + t \overline{y} + z. $$
Recall also that we have a $4$--coloring of the edges of $\Ec$ defined in Section \ref{ss:simple}, and at each vertex there are only two triangles of two different colors meeting.

\begin{figure}
[hbt] \centering
\includegraphics[height=3 cm]{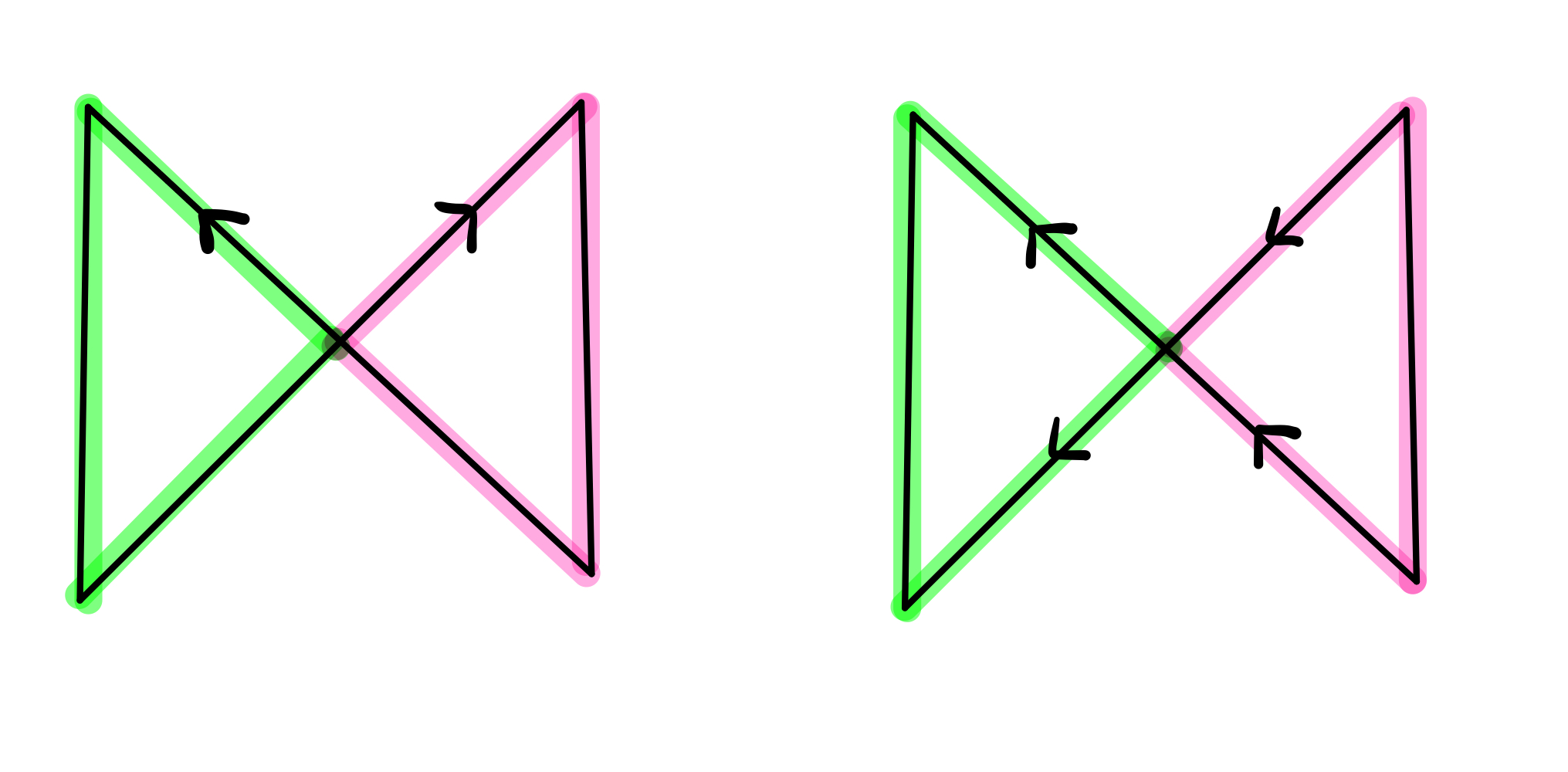}
\caption{A vertex which is a fork (left) and one which is not (right). The two arrows pointing away from the vertex have to be in different triangles to get a fork.}
\label{fig:fork1}
\end{figure}

\begin{figure}
[hbt] \centering
\includegraphics[height=6 cm]{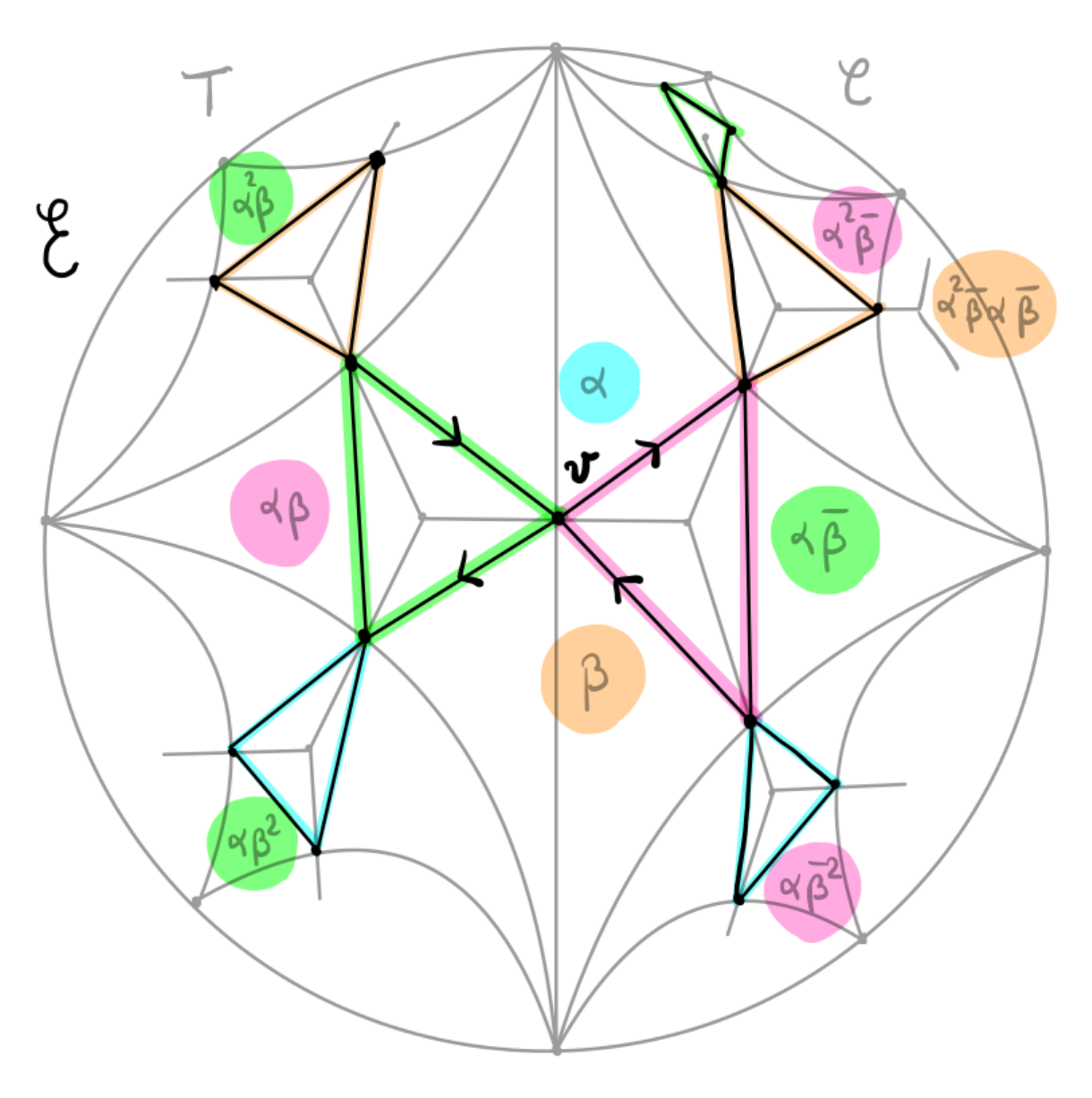}
\caption{A fork with the labeling of the neighboring regions}
\label{fig:fork2}
\end{figure}

\begin{Definition}
We say that a vertex $v  \in \Ec^{(0)}$ is a {\it fork} for $\rho$ if there are (at least) two edges oriented by $\rho$ having $v$ as a source and are of different colors (for the four-coloring). Visually, $v$ is a fork if two arrows point away from $v$ in two different triangles. See Figure \ref{fig:fork1}.

 Equivalently, the vertex $v = (X , Y ; Z, T) \in \mathcal{E}^{(0)}$ is a {\it fork} for $\rho$ if the following conditions hold:
$$\begin{cases}|\Tr\rho(XY^{-1}) | &< \max \{ |\Tr \rho(XY^2) | , |\Tr \rho (X^2Y) | \} \\
 &\mbox{ and } \\
  |\Tr \rho(XY)| &< \max \{ |\Tr \rho(XY^{-2}) | , |\Tr \rho (X^{-2}Y) | \}.\end{cases}$$

In terms of the four traces $(x,y,z,t)$ these conditions can be written as :
$$\begin{cases} |t| < |t- x \overline y + yz| \\ or \\ |t| < |t - x \overline{y} + \overline{xz}|, \end{cases} \quad \mbox{ and } \begin{cases} |z| < |z-xy+\overline{xt}| \\ or \\ |z| < |z - xy + \overline{y}t |. \end{cases} $$
We will call these inequalities the \emph{fork conditions}. See Figure \ref{fig:fork2}.
\end{Definition}

\begin{Lemma}\label{fork}
Let $c \in \C$, $[\rho] \in \X_c$, and $v = (X,Y) \in \Ec^{(0)}$. If $v$ is a fork for $\rho$, then we have that $\min \{|x| , |y| \} \leq \max\{6,\sqrt{|\Re(c)|}\}$.
\end{Lemma}
	
\begin{proof}
Without loss of generality, and by using the fork conditions, we can assume that $|t| > | t - x\bar y +yz |$ and $|z| > |z -xy +\overline{xt}|$. Let $C = \max\{6,\sqrt{|\Re(c)|}\}$ and assume by contradiction that $\min \{|x|,|y|\} > C \geq 6$. 
	
Note that $|t - x\bar y +yz | \geq |yz-x\bar y|-|t| \geq  |y||z|-|x||y| - |t|$. So, since we have $|t| > | t - x\bar y +yz|$ by the fork assumption, we deduce that $2|t| +|x||y| > |y||z|$, and hence we obtain: 
	$$|z| < \frac{2 |t|}{|y|} + |x|.$$
Similarly, as $|z -xy +\overline{xt} | \geq |\overline{tx}-xy|-|z| \geq |t||x|-|x||y|- |z|$ and $|z| > |z -xy +\overline{xt}|$, we get: 
						$$|t| < \frac{2 |z|}{|x|} + |y|.$$
We combine both inequalities to get
	\begin{align*}
		|z|+|t| & <  2 \left( \frac{|t|}{|y|} + \frac{|z|}{|x|} \right) + (|x| + |y|)\\
		& < 2 \frac{|t| + |z|}{C} + 2\max\{ |x| , |y| \}.
	\end{align*}
From this it follows that $$|z|+|t| < \frac{2C}{C-2} \max \{|x| , |y|\}.$$
	
Recall that we have 
	$$P=|x|^2 |y|^2 + |x|^2 + |y|^2 + |z|^2 + |t|^2 - 2 \Re (xy\overline z ) - 2 \Re (\overline{x} y t ) - 3.$$
	
Since $C \geq 6$, we have that $\frac{4C}{C-2} \leq C < \min\{|x|,|y|\}$. Hence, using the inequality above, we get:
\begin{eqnarray*}
|2 \Re (xy\overline z ) + 2 \Re (x \overline y t )| \leq 2 |xy| (|z| + |t|) &<& \frac{4C}{C-2} |xy| \max \{|x| , |y|\}\\ &<& |xy| \min \{|x| , |y|\} \max \{|x| , |y|\} = |x|^2 |y|^2.\end{eqnarray*}
Moreover, using Lemma \ref{bound}, we can see that $$|z|^2 + |t|^2 -3 \geq \max (|z| , |t|)^2 -3 \geq \frac{C^2}{16} \max(|x|, |y|)^2 - 3 \geq \frac{C^4}{16}-3 > 0.$$
This implies that:$$P=|x|^2 |y|^2 + |x|^2 + |y|^2 + |z|^2 + |t|^2 - 2 \Re (xy\overline z ) - 2 \Re (x \overline y t ) - 3  > |x|^2+|y|^2 > 2 C^2.$$  
This gives a contradiction, as $P = 2 \Re (c) \leq  2 C^2$. 
	\end{proof}

As we will see, the conclusion of the Fork Lemma (Lemma \ref{fork}) still holds for vertices that satisfy the fork conditions ``up to epsilon,'' and for certain arguments we will need the notion of an $\varepsilon$--fork. More precisely, we give the following definition.

\begin{Definition}
Let $\varepsilon >0$. A vertex $v \in \Ec$ is an $\varepsilon$--{\it fork}, if changing the value of one of the four regions around $v$ by less than $\varepsilon$ makes it a fork. This means that $v$ has (at least) two arrows pointing away from it, or pointing towards it but such that the difference between the values at the opposing regions of the edge, say $x, x' \in \C$ are such that $\vert \vert x \vert - \vert x' \vert \vert < \varepsilon$. 
\end{Definition}

The Fork Lemma (Lemma \ref{fork}) above can be generalized to $\varepsilon$--forks.

\begin{Corollary}
	Let $c \in \C$, $[\rho] \in \X_c$, and $v = (X,Y) \in \Ec^{(0)}$. Let $C = \max\{6,\sqrt{|\Re(c)|}\}$ and $\varepsilon_0 = \varepsilon_0(C) = C^2 \frac{C-6}{4} > 0$.  For all $\varepsilon < \varepsilon_0$, if $v $ is an $\varepsilon$--fork, then $\min \{ |x|, |y| \} < C$.
	\end{Corollary}

\begin{proof}
We will follow an argument similar to the one used to prove the Fork Lemma (Lemma \ref{fork}). Assume by contradiction that $\min \{|x|,|y|\} > C \geq 6$. Let $\varepsilon < \varepsilon_0$. Assume $v$ is an $\varepsilon$--fork. We keep the same notation for $x,y,z,t$ as before. The fork conditions give 
	$|t| > |t - x \overline y + yz | - \varepsilon $ and $|z| > |z - xy + \overline{xt} | - \varepsilon$. The same steps as before give:
	$$|z| < 2\frac{|t|}{|y|} + |x| + \frac{\varepsilon}{C} \quad \mbox{ and } \quad |t| < 2 \frac{|z|}{|x|} + |y| + \frac{\varepsilon}{C}.$$
	
	We combine both inequalities and obtain $$|z| + |t| < \frac{2C}{C-2} \max \{ |x| , |y| \} + 2 \frac{\varepsilon}{C-2}.$$
	As $\varepsilon < \varepsilon_0$, we obtain directly that $|z|+|t| < \dfrac{|C|}{2} \max\{|x|,|y|\}$, since $C < \max\{|x|,|y|\}$.
	
	Now, using the hypothesis  $C < \min \{ |x| , |y| \}$, we get:
	$$|2 \mathfrak{R} (xy\overline z) + 2 \mathfrak{R} (\overline{x} y t) | \leq 2 |xy| (|z|+|t|) < |x|^2 |y|^2.$$
	
	Using reasoning similar to the previous proof, we can then see that $P > 2C^2$, which gives a contradiction. So we have that $\min \{ |x| , |y| \} \leq C$, as we wanted to prove.
\end{proof}

\subsubsection{Connectedness}

Recall that we have the following identification $\Omega = \mathcal{C}^{(0)} = \mathcal{S}$. 

For any $K \geq 0$, and $[\rho] \in \X$, we denote $\Omega_\rho (K)$ the set of regions with traces smaller than $K$:
$$\Omega_\rho (K) = \{ \gamma \in \Omega \mid  |\tr (\rho (\gamma))| \leq K \}\subset \HH.$$
Since every element $\gamma \in \Omega$ correspond to the closure of a region in the complement of the (properly embedded trivalent) tree $\mathcal{T}$ in $\HH$, as defined in Section \ref{ss:simple}, we consider the set $\Omega_\rho (K)$ as a subset of $\HH$. Using the Fork Lemma, in this section we deduce that, for $K$ large enough, these sets are connected when represented as subsets of $\HH$.

\begin{Lemma}
	Let $c \in \C$, and $[\rho] \in \X_c$. For all $C \geq  \max\{6,\sqrt{|\Re(c)|}\}$, the set $\Omega_\rho (C)$ is connected.
\end{Lemma}

\begin{proof} Assume by contradiction that $\Omega_\rho (C)$ is not connected. We make an argument based on the minimal distance between two connected components of $\Omega_\rho (C)$. Recall that if $v = (X,Y ; Z, T) \in \mathcal{E}^{(0)}$, then we have :
$$P=|x|^2 |y|^2 + |x|^2 + |y|^2 + |z|^2 + |t|^2 - 2 \Re (xy\overline z ) - 2 \Re (\overline{x} y t ) - 3 = 2 \Re (c),$$ which implies $P \leq 2C^2$.

If the distance is $1$. Then we have a vertex $v = (X , Y; Z,T) \in \mathcal{E}^{(0)}$ such that $|z| , |t| \leq C$ and $|x| , |y| > C$. In that case we have:
$$|2 \Re (xy\overline z) + 2 \Re (\overline x y t)| < 4C |xy| = 5C |xy| - C|xy| < |x|^2 |y|^2 - 3,$$
since $5C < C^2 < |x||y|$ and since $C |x||y| \geq C^3 > 3$. This gives: $$P > |x|^2 + |y|^2 \geq 2C^2,$$ which is a contradiction.

If the distance is $2$, then we have an edge $(X,Y,T ; Z ,Z')$. We can assume, without loss of generality, that $|z| \leq  C$, $|z'| = |z+\overline{yt} - xy| \leq C$, and $|x|, |y|, |t| >C $. In that case the inequality $|z+\overline{y}t - xy| \leq C$ implies that $C < |t| \leq \frac{C+|z|}{|y|} + |x| < |x| + 2.$ This gives: $$|2 \Re (xy\overline z) + 2 \Re (\overline x y t)|< 2C |xy| + 2 |xy| (|x| + 2) < 2 |xy| (C+|x|+2)< |x|^2 |y|^2,$$
where the last inequality follows from $$2 (C+|x|+2) = |x| \left(2 + \frac{4}{|x|}+\frac{2C}{|x|}\right) < |x| \left(2 + \frac{3C}{|x|}\right) < |x| 5 < |x| C < |x||y|.$$ Since $|t|^2-3> C^2-3 >0$ and $|x|^2 + |y|^2 +|z|^2> 2 C^2$, we also find that $P> 2C^2$, which gives the same contradiction.

If the distance is greater than or equal to $3$, then we can consider the shortest path between these components and orient the edges along this path according to $\rho$. The first and last oriented edges are pointing in opposite directions (each towards one of the connected components). So there exists a vertex in the middle of the path that is a fork for $\rho$. So there is a region with trace less than $C$ with a smaller distance to each connected component, which contradicts the minimality of the distance.
	
\end{proof}

	\subsection{Escaping rays and Neighbors}\label{escaping}
		Let $P$ be a geodesic ray in the edge graph, starting at a vertex $v_0$ and consisting of edges $e_n$ joining $v_n$ to $v_{n+1}$. Such a geodesic only passes once through each triangle of the edge graph. We say that such an infinite geodesic ray is an \emph{escaping ray} if for all $n \in \N$, each edge $e_n$ is directed from $v_n$ to $v_{n+1}$. Given a representation $[\rho] \in \mathfrak{X}_c$, an infinite geodesic ray is called an \emph{escaping ray for $\rho$} if it is an escaping ray with respect to the orientation induced by $\rho$, see Section \ref{orientation}.

		We will prove that the existence of escaping rays for a given representation $\rho$ will imply that either the ray will eventually be contained in the boundary of a region $X \in \mathcal{C}^{(0)}$  such that $\rho (X)$ only has eigenvalues on the unit circle (which is equivalent to $f(x) \leq 0$ according to Theorem \ref{corol_f} ), or there are infinitely many regions with value bounded by a constant depending only on the boundary  parameter $c \in \C$.
		 
		First we consider the case of an escaping ray that is included in the boundary of a given region in $\mathcal{C}^{(0)}$ , and then we will describe the general case. 

	\subsubsection{Neighbors around a region}

	In this subsection we are interested in the behavior of the values of the traces of regions that are all neighbors of a central one.  Let $\a \in F_2$. The neighbors of the region associated to $\a$ are all associated to elements of the form  $\g_n = \a^n \b$ for all $n \in \Z$. 
	
	Let $c \in \C$ and let $[\rho] \in \X_c$ be a representation. Let $x = \tr (\rho (\a))$ and consider the sequence defined by $u_n = \tr (\rho (\g_n))$ for all $n \in \Z$. The trace identity gives the following relation:
	\begin{equation}\label{u_n}
		\forall n \in \Z, \ u_{n+3} = x u_{n+2}  - \overline{x} u_{n+1} + u_n .
	\end{equation}
	
	To study the solutions of this recurrence relation we define $X_n = \begin{pmatrix} u_{n+2} \\ u_{n+1} \\ u_n \end{pmatrix} \in \mathbb{C}^3$ and consider the matrix $M_x = \begin{pmatrix} x & - \overline x & 1 \\ 1 & 0 & 0 \\ 0 & 1 & 0 \end{pmatrix}$, so that we have $X_{n+1} = M_x X_n$.  The matrix $M_x$ is conjugate to a matrix in $\rm{SU} (2,1)$ and its characteristic polynomial is given by $\chi_x(s) = s^3 - x s^2 + \overline x s - 1$. We can use Corollary \ref{corol_f} to determine the behavior of the eigenvalues of $M_x$, using the function $f$ defined by $f(x) = \mathrm{Res} (\chi_x , \chi_x')$, and hence of the sequence $(u_n)_{n\in \Z}$ depending on the value of $f(x)$. 
	
	If $f(x)= 0$, then the coefficients of the matrix $M_x^n$ grow at most quadratically in $n$, and hence the growth of the sequence $(X_n)_{n\in \Z}$, and subsequently of the sequence $(u_n)_{n\in \Z}$, is at most quadratic, see Appendix \ref{app:fzero}.
	
	When $f(x) \neq 0$, we can diagonalize the matrix $M_x$ as $M_x = P D P^{-1}$ with 
	$$P = \begin{pmatrix} \lambda & \frac{\bar \lambda}{\lambda} & \frac{\, 1\, }{\bar \lambda} \\ 1 & 1 & 1 \\ \frac{1}{\lambda} & \frac{\, \lambda \,}{\bar \lambda} & \bar \lambda \end{pmatrix}, \quad D = \begin{pmatrix} \lambda & 0 & 0 \\ 0 & \frac{\bar \lambda}{\lambda} & 0 \\ 0 & 0 &   \frac{\, 1\, }{\bar \lambda}\end{pmatrix}, \;\; \text{ and } \;\; x = \lambda + \frac{\bar \lambda}{\lambda} + \frac{\, 1\, }{\bar \lambda}.$$
	
Hence we get $X_n = P D^n P^{-1} X_0,$ where $X_0 = \begin{pmatrix} u_2 \\ u_1 \\ u_0 \end{pmatrix}$. 

Let $(e_1, e_2, e_3)$ be the basis of $\C^3$ formed by the columns of the matrix $P$, which correspond to the eigenvectors of the matrix $M_x$. We can decompose $X_0 = m_1 e_1 + m_2 e_2 + m_3 e_3$ with $(m_1, m_2, m_3) \in \C^3$. In that case, we get that:
	
	$$X_n = m_1 \lambda^n e_1 + m_2 \left( \frac{\bar \lambda}{\lambda} \right)^n e_2 + m_3 \left( \frac{1}{\bar \lambda} \right)^n e_3.$$
	
	If $f(x)< 0$, then $|\lambda| = 1$ and we see that the sequence $(X_n)_{n\in \Z}$ stays bounded in $\C^3$, and subsequently the sequence $(u_n)_{n\in \Z}$ is bounded as the projection of $X_n$ on $\C \cdot e_3$. 
	
	When $f(x) > 0$, then we have $|\lambda| > 1$. If $m_1 m_3 \neq 0$ we see that the sequence $X_n$ grows exponentially in both directions. However, it is possible to have $m_1 m_3 = 0$ and in that case the sequence will stay bounded in one direction. The purpose of the next lemma is to prove that the bound on such a sequence will only depend on the boundary parameter $c$.
	
	\begin{Lemma}\label{prop:bounded_neighbors}
		There exists a constant $D = D(c)$, such that if the sequence $(u_n)_{n\in \N}$ defined above with $f(x) >0$ does not grow exponentially in both directions, then there are infinitely many terms of the sequence such that $|u_n| < D$. 
	\end{Lemma}

	\begin{proof}
	
	With the notation above, we are necessarily in the case $m_1 m_3 = 0$. We can assume, without loss of generality, that $m_1 = 0$,  so that we have 
	$$X_0 = m_2 e_2 + m_3 e_3 \;\text{ and }\;\;X_n = m_2 \left( \frac{\bar \lambda}{\lambda} \right)^n e_2 + m_3 \left( \frac{1}{\bar \lambda} \right)^n e_3.$$
	
	Since $|\lambda| > 1$, we see that the sequence $X_n$ will accumulate on a circle of radius $|m_2 |$ in the complex line $e_2$. More precisely, for all $\varepsilon >0$, there exists $n_0 \in \N$, such that for all $n \geq n_0$, there exists $\theta_n \in \R$ such that $\Vert X_n - |m_2| e^{i\theta_n} e_2 \Vert < \varepsilon$. In other words, for $n$ large enough, the vector $X_n$ gets arbitrarily close to the set $\{ |m_2| e^{i\theta}  e_2 , \theta \in \R \}$.
	
	Hence, to prove the lemma, it suffices to find a uniform bound on $|m_2|$ that only depends on the boundary parameter $c$. Let $\varepsilon$ be arbitrarily small, and $n \geq n_0$ as above. 
	
By definition of the sequence, we have :
 $$x = \tr (\rho (\a)), \ u_n = \tr (\rho (\alpha^{n} \beta)), \ u_{n+1} = \tr (\rho (\alpha^{n+1} \beta )), \text{  and  } u_{n+2} = \tr ( \rho (\alpha (\alpha^{n+1}) \beta)).  $$
 
 	By denoting $\beta' = \alpha^{n+1} \beta$ we can write these values as :  	
 	$$x = \tr (\rho (\a)),  \quad u_{n+1} = \tr (\rho ( \beta' )), \quad u_{n+2} = \tr ( \rho (\alpha  \beta' )), \quad \overline{u_n} = \tr (\rho (\a \beta'^{-1} )),$$ so we can identify the quadruple $(x, u_{n+1} , u_{n+2} , \overline{u_n})$ as the character of a representation in $\X_c$. This means that the quadruple $(x, u_{n+1} , u_{n+2} , \overline{u_n})$ will satisfy the equations for $P$ and $Q$ with $y=u_{n+1} , z = u_{n+2}$ and $t = \overline{u_n}$. 
	
	Using our hypothesis on $n$, this quadruple is arbitrarily close to a quadruple of the form $\left(x, |m_2| e^{i \theta}, |m_2[ e^{i\theta} \frac{\bar \lambda}{\lambda} , |m_2| e^{-i\theta} \frac{\bar \lambda}{\lambda} \right),$ for some $\theta \in \R.$ 
	
	By continuity of the expression relating $P$ and $Q$ with the quadruple of trace coordinates, we know that this ``limit'' quadruple also satisfies the  equations for $P$ and $Q$ and corresponds to an equivalence class of a representation in $\X_c$.

	We can first write the equation for $P$ with the quadruple $\left(x, |m_2| e^{i \theta}, |m_2[ e^{i\theta} \frac{\bar \lambda}{\lambda} , |m_2| e^{-i\theta} \frac{\bar \lambda}{\lambda} \right)$, which gives : 
	\begin{align*}
		P & = |x|^2 |m_2|^2 + |x|^2 + 3 |m_2|^2 - 2 \Re \left(x |m_2| e^{i\theta} \displaystyle\frac{\, \lambda \,}{\bar \lambda} |m_2|e^{-i\theta} \right) - 2 \Re \left( \overline{x} |m_2|e^{i\theta} \frac{\, \bar \lambda\, }{\lambda} |m_2|e^{-i\theta} \right) -3 \\
		& = |m_2|^2 \left( |x|^2 - 4 \Re \left( x \frac{\lambda}{\, \bar \lambda\, } \right) +3\right) - 3 + |x|^2.
		\end{align*}
		
So we obtain an expression for $|m_2|$ in terms of $P$ and $x$:
		\begin{equation}\label{eqn:B} |m_2|^2 = \dfrac{P + 3 - |x|^2}{|x|^2 - 4 \Re \left( x \frac{\, \lambda\, }{\bar \lambda} \right) +3}.\end{equation}

The equation for $Q$ gives :	
	\begin{align*}
		&Q =  |m_2|^4 \left(5 + |x|^2 + 2 \Re \left( 2 \left(\frac{\bar \lambda}{\lambda}\right)^3- 2 \overline x  \frac{\bar \lambda}{\lambda} - x  \left(\frac{\bar \lambda}{\lambda}\right)^2 \right) \right) \\
		& + |m_2|^2 \left( |x|^4+3|x|^2-18+2\Re \left( 2 x^2 \frac{\bar \lambda}{\lambda}-3 x  \left(\frac{\bar \lambda}{\lambda}\right)^2+6 \overline x  \frac{\bar \lambda}{\lambda} - 2 |x|^2 \overline x  \frac{\bar \lambda}{\lambda}+\overline x^2 \left( \frac{\bar \lambda}{\lambda}\right)^2-x^3 \right) \right) \\
		& + 9 - 6|x|^2+2\Re (x^3).
\end{align*}
	
Using the expression of $|m_2|^2$ found above and replacing $x = \lambda + \frac{\bar \lambda}{\lambda} + \frac{1}{\bar \lambda}$ and writing $\lambda = r e^{is}$ with $r >0$ and $s \in \R$, we can obtain (for example, using {\it Mathematica}) the following formula:
\begin{equation}\label{eqn:Q} Q = \dfrac{(Pr^4 - r^6 + 1)(Pr^2 + r^6 - 1)}{r^4 (r^2+1)^2}.\end{equation}

This means that $r$ is one of the finitely many solutions of the equation above. At this point, we have already proved that if the sequence $(u_n)_{n\in \N}$ does not grow exponentially, then $|\lambda|$ and hence $|x|$ can only take finitely many values that only depend on $P$ and $Q$. Moreover, we cannot have $r=1$ as a solution as in that case we would have $f(x) = 0$ which is excluded from our consideration.

Finally, if we look at Equation \ref{eqn:B} we see that its denominator can be written as:

$$|x|^2 - 4 \Re \left( x \frac{\lambda}{\bar \lambda} \right) +3 = |\lambda|^2 - 2 \Re \left( \frac{\lambda^2}{\bar \lambda}\right) - 2 \Re \left(\frac{\lambda}{\bar{\lambda}^2}\right) + 2 + \frac{1}{|\lambda|^2} =\dfrac{(|\lambda|^2+1)|\lambda^2 - \bar{\lambda} | ^2 }{|\lambda|^4},  $$
which makes it clear that the denominator of the expression can only be equal to $0$ when $\lambda \in \zeta_3$ and hence when $|\lambda| = 1$. This means that if $r = |\lambda|$ is one of the finitely many solutions of Equation \eqref{eqn:Q}, then the expression of $|m_2|^2$ is bounded by a value that depends only on $P$ and $Q$, and hence that depends only on $c$. 

We can conclude that there exists $D = D(c)$, only depending on $c$, such that $|m_2| < D$. Thus, if the sequence $(u_n)_{n\in \Z}$ does not grow exponentially for $n \to + \infty$, then $|u_n|$ will converge to a value smaller than $D$.
\end{proof}
	
\begin{Remark} We can use the previous proof to produce an algorithm to find the constant $D$.  Precisely, first solve for $r^2$ using Equation \eqref{eqn:Q} and then plug it into Equation \eqref{eqn:B} to obtain a formula for $|m_2|$ in terms of $P,Q$ and $s$.  A value for $D$ can be taken to be $|m_2|+\epsilon$.  So we need to find the maximum of $|m_2|$ on the circle (parametrized by $s$), which can always be done numerically.
\end{Remark}

For the rest of the paper, for all $c \in \C$, we set 
\begin{equation}\label{def:Mc}
M(c):=\max \{6, \sqrt{|\Re (c)|} , D(c)\}. \end{equation}

Lemma \ref{prop:bounded_neighbors} implies the following result.

\begin{Lemma}\label{functionH}
Let $c \in \C$. There exists a function $H_c \co \C \rightarrow \R \cup \{ \infty \}$ with the following properties, for all $x \in \C$: 
\begin{enumerate}
\item[$(1)$] $H_c(x) = H_c(\overline x)$.
\item[$(2)$] $H_c(x) \geq M(c)$.		
\item[$(3)$]  $H_c(x) = \infty$ if and only if $f(x)\leq 0$. 
\item[$(4)$] If $f(x) >0$, then for any sequence $u_n$ as above satisfying Equation \eqref{u_n} and coming from a representation $[\rho] \in \X_c$, there exists $n_1, n_2 \in \Z \cup \{ \pm \infty \}$ such that:
\begin{enumerate}
\item[$(i)$] $|u_n| \leq H_c(x)$ if and only if $n_1 \leq n \leq n_2$;
\item[$(ii)$] $|u_n|$ is monotonically decreasing for $n < n_1$ and increasing for $n > n_2$. 
\end{enumerate}
\end{enumerate}
\end{Lemma}

There are similar statements in previous works (\cite{bow_mar} 3.14,  \cite{tan_gen} 3.20, \cite{mal_ont} 4.15, and \cite{MaPa1} 3.4.1) that all rely on the same basic principle. We refer the interested reader to the proofs given in these papers that are in simpler cases. 
\begin{proof} We will omit certain computational details for clarity.

First, we simply set $H_c (x) = \infty$ when $f(x) \leq 0$.

	Now, let $x \in \C$, such that $f(x) > 0$. Assume that we have $[\rho] \in \X_c$, a region $X \in \Omega$ such that $x = \tr (\rho (X))$ and the sequence $(u_n)$ defined as before. So $(u_n)$ satisfies Equation \eqref{u_n} and is such that $(x, u_1, u_2, \overline{u_0} )$ is the character of the representation $\rho$. As discussed above, there exists $(m_1, m_2, m_3) \in \C^3$ such that  
	$$u_{n+1} = m_1 \lambda^n + m_2 \left( \frac{\bar{\lambda}}{\lambda} \right)^n + m_3 \left( \frac{1}{\bar \lambda} \right)^n.$$

	Assume first that $m_1 m_3 \neq 0$. Up to reindexing the sequence we can assume that $|m_3| \leq |m_1| \leq |\lambda|^2 |m_3|$. The quadruple $(x,u_1, u_2, \overline{u_0} )$ satisfies the equations for $P$ and $Q$, so using an argument similar to the one for finding the bound on $|m_2|$ in the proof of Lemma \ref{prop:bounded_neighbors}, we can show that there exists a constant $K_c (x)$ such that  $\frac{|m_2|}{\sqrt{|m_1m_3|}}$ and $|m_1 m_3|$ are both less than $K_c (x)$. One can see that 
	\begin{align*}
		\frac{|u_{n+2}| - |u_{n+1}|}{\sqrt{|m_1 m_3|}} &\geq |\lambda|^n \sqrt{\frac{|m_1|}{|m_3|} } ( |\lambda| - 1) - 2 \frac{m_2}{\sqrt{|m_1 m_3|}} - \sqrt{\frac{|m_3|}{|m_1|}} \frac{1}{|\lambda|^{n+1}} (|\lambda|+1) \\
		& \geq |\lambda|^n (|\lambda|-1) - 2 K_c(x) - (|\lambda|+1).
	\end{align*}
	
	Hence, if $|\lambda|^n$ is larger than some constant $L_c (x):= \frac{2K_c (x) + (|\lambda| + 1)}{|\lambda|-1}$, then we have that $|u_{n+2}| \geq |u_{n+1}|$. 
	
	On the other hand, as $|u_{n+1}| \leq |m_1| |\lambda|^n + |m_2| + |m_3| \frac{1}{|\lambda|^n}$, we infer that there exists $H_c(x)$ such that if $|u_{n+1}| \geq H_c (x)$, then $|\lambda|^n \geq L_c (x)$. Hence, for indices such that $|u_n| \geq H_c (x)$ we have that the sequence is monotonically increasing, which is the desired condition.
	
	We can choose $H_c (x)$ to be always bigger than $M(c)$ so that in the case $m_1 m_3 = 0$, Lemma \ref{prop:bounded_neighbors} gives the desired result. The sequence $u_{-n}$ behaves in a similar way, simply replacing $x$ with $\bar x$, and hence this shows that $H_c( x) = H_c (\bar x)$.  
	\end{proof}
	
\subsubsection{Escaping rays}
	
Using the constant $M(c)$ defined in (\ref{def:Mc}) we now prove : 

\begin{Lemma}\label{lem:escaping}
	Let $c \in \C$ and $[\rho] \in \X_c$. Set $M(c):=\max \{6, \sqrt{|\Re (c)|} , D(c)\}$. Suppose that $\{ e_n \}_{n\in \N}$ is an escaping ray. Then there are two cases: 
	\begin{enumerate}
		\item[$(1)$] there exists region $\a$ such that the ray is eventually contained in the boundary of $\a$, such that $f (x ) \leq 0$ with $x = \tr (\rho (\a))$, or 
		\item[$(2)$] the ray meets infinitely many distinct regions with trace smaller than $M(c)$.
	\end{enumerate}
\end{Lemma}

\begin{proof}
	Let $(e_i)_{i\in \N}$ be an escaping ray. First, assume that the ray is eventually contained in the neighborhood of some region $\a$, and note $x = \tr (\rho (\a))$. If  $f(x) \leq 0$, then we are in the first case. Otherwise, we have $f(x) >0$. Note that the sequence $(|u_n|)_{n\in \N} $ of traces of regions around $x$ cannot be exponentially growing with respect to $n$, or else the ray would not be. From the previous lemma, it means that there are infinitely many distinct regions with trace smaller than $D (c)$.
	
Now assume that the escaping ray $(e_i)_{i \in \N}$ does not eventually end in the neighborhood of a given region. We consider a $4$-coloring of the complex $\mathcal{E}$ as described in Section \ref{ss:simple}. Let $(X_i)_{i\in \N}, (Y_i)_{i\in \N}, (Z_i)_{i\in \N}, (T_i)_{i\in \N}$ be the four sequences of regions corresponding to the four regions around the head of edge $e_i$, with the letters $X,Y,Z,T$ each corresponding to a given color. Let $x_i = \tr (\rho (X_i))$; respectively, for $y_i, z_i, t_i$.
	
	 As the escaping ray does not end in the neighbors of a given region, we see that the four sequences $|x_i|, |y_i|, |z_i|, |t_i|$ are infinite, decreasing, and bounded below. So the modulus of each sequence is converging towards some limit.
	
	For all $\e >0$, there exists $n_0$ such that, for all $n \geq n_0$, the vertex $v_n$ between the two consecutive directed edges  $e_n$ and $e_{n+1}$ is an $\e$-fork. For $\e$ small enough, we can then apply the $\e$-fork Lemma to conclude that for all $n \geq n_0$ there is at least one region around the vertex $v_{n}$ that has trace smaller than $C:= \max \{6 , \sqrt{|\Re (c)|} \}$. As the escaping ray is not contained in the neighborhood of a given region, we know that there will be infinitely many distinct such regions. 
\end{proof}

\subsection{Bowditch representations}\label{definition}

We can now define Bowditch representations. 
	
		\begin{Definition}\label{def_bow}
			Let $c \in \C$, and let $\rho \co F_2 \rightarrow \SUtwo$ be a representation such that $[\rho]\in\X_c$.  We say that $\rho$ is a {\it Bowditch representation} if
\begin{enumerate}
				\item $\forall \gamma \in \Sc$, we have that $\rho(\gamma)$ is loxodromic, and
				\item the set $\Omega(M) = \{ \gamma \in \Sc \ | \ \tr (\rho (\gamma)) < M \}$ is finite,
\end{enumerate}
where $M = M(c)$ as defined in (\ref{def:Mc}). 

We denote by $\mathfrak{X}_{BQ,c} = \mathfrak{X}_{BQ, c}(F_2 , \SUtwo)$ the set of conjugacy classes of representation in $\X_c$ that are Bowditch, and more generally, we denote  $$\mathfrak{X}_{BQ} = \bigcup_{c \in \C} \mathfrak{X}_{BQ, c}(F_2 , \SUtwo).$$ 
		\end{Definition}
		
		The following result is an immediate consequence of Lemma \ref{prop:bounded_neighbors} and Lemma \ref{lem:escaping}.
		\begin{Lemma}
			If $[\rho] \in \mathfrak{X}_{BQ}$, then there will be no escaping rays for $\rho$. Moreover, for any region $X \in \Omega$ and neighboring regions $\{Y_n\}$ around $X$, the absolute value of traces $|\tr\left(\rho(Y_n)\right)|$ will have exponential growth in both directions.
		\end{Lemma}

	\subsection{Attracting subgraph} \label{sec:tree}
	
		\subsubsection{Attracting arc}
		
		Recall that we have the following identifications: $\Omega = \mathcal{C}^{(0)} = \mathcal{S}$. We will make use of Lemma \ref{functionH} to define an attracting subgraph in the neighborhood of certain regions. 
		
Let $c \in \C$, $[\rho] \in \mathfrak{X}_c$ and $K \geq M(c)$. For each region $X \in \Omega (K)$, we define:
		$$J_\rho (K ,X) := \bigcup \left\{ e \in \mathcal{E}^{(1)} | \ e = (X, Y_n, Y_{n+1}) \mbox{ and } |y_n| , |y_{n+1}| \leq \max\{H_c(x) , K\} \right\}.$$
		 This is a connected nonempty subarc in $\partial X$ with the following properties:
			\begin{enumerate}
				\item Every directed edge $\vec e_\rho$ such that $e \in \partial X$ that is not in $J_\rho (K, X)$ points towards $J_\rho (K, X)$.
				\item If $Y \in \Omega(K)$ and $e \in \partial X$ has one vertex $v = (X, Y) \in \mathcal{E}^{(0)}$, then $e \in J_\rho (K, X)$. 
			\end{enumerate} 
			
	This definition is similar to the definitions used in Tan-Wong-Zhang \cite{tan_gen} and Maloni-Palesi-Tan \cite{mal_ont}. However, in our setting, we need to define a subgraph that is slightly larger than this subarc, but with similar properties, because of the way we define a subtree to be ``attracting." Recall that the graph $\mathcal{E}$ is a ``trivalent tree of triangles" in the sense that it is a graph made of triangles, one for each triangle $T$ in $\mathcal{C}^{(2)}$. We want to define a bigger set $\widetilde{J}_\rho (K, X) \subset \mathcal{E}^{(1)}$ with the property that, if one edge $e$ is in $J_\rho (K, X)$, we want to add, in $\widetilde{J}_\rho (K, X)$, the whole triangle that $e$ is part of. It is clear that $J_\rho (K, X) \subset \widetilde{J}_\rho (K, X)$. We say that a triangle $\Delta \subset \mathcal{E}^{(1)}$ is in the boundary $\partial X$ of the region $X$ if one of its edges has this property. So we define the set $\widetilde{J}_\rho (K, X)$ as a connected union of triangles in $\partial X$ satisfying the following properties:
		\begin{enumerate}
			\item For each triangle $\Delta$ in $\partial X$ such that $\Delta$ is not in $\widetilde{J}_\rho (K,X)$, then for each vertex $v$ of $\Delta$, the directed edge from $v$ forming a geodesic toward $\widetilde{J}_\rho (K, X)$ is pointing towards $\widetilde{J}_\rho (K, X)$.\footnote{Note that for one edge of $\Delta$ this is an empty condition.}
			\item If $Y \in \Omega (K)$ and $\Delta$ is a triangle adjacent to $X$ and $Y$, then $\Delta$ is in $\widetilde{J}_\rho (K, X)$. 		 
		\end{enumerate}
	
	We can now explain better why we need to define this extension. Let $e$ be an edge which is part of a ``triangle" $\Delta$. It could happen that $e\in J_\rho (K, X)$ but no other edge of $\Delta$ is. In this situation, the geodesics from the vertex of $\Delta\setminus e$ to $J_\rho (K, X)$ are not unique and it could happen that it is not true any longer that both those geodesic paths point toward $J_\rho (K, X)$, as we will require in the definition of $\rho$--attracting below.   
	
		\subsubsection{Attracting subgraph}
	
\begin{Definition}
	A non-empty subgraph $T \subset \mathcal{E}$ is said to be $\rho$--{\it attracting}, if for every vertex $v\in \mathcal{E}^{(0)}$ such that $v \notin T$ and for every geodesic in $\mathcal{E}$ from $v$ to $T$, each $\rho$--oriented edge along the geodesic points towards $T$. 
\end{Definition}		
	
For all $K \geq M(c)$ we constructed a subarc $\widetilde{J}_\rho (K, X)$ for all regions $X \in \Omega = \Sc$. We can now define the graph $$T_\rho (K) := \bigcup_{X \in \Omega (K)} \widetilde{J}_\rho (K, X).$$
This set is a subgraph of $\mathcal{E}$ and we can prove the following result.
	\begin{Proposition}\label{attracting}
		If $K \geq M(c)$, then the subgraph $T_\rho(K)$ is connected and $\rho$-attracting.
	\end{Proposition}
	
	\begin{proof}
		 First, let us show that $\mathcal{T} = T_\rho (K)$ is connected. Let $e, e'$ be two edges of $\mathcal{T}$. Then $e$ and $e'$ belong to a triangle adjacent to  regions $X, X' \in \Omega(K)$. As $\Omega(K)$ is connected, we can find a sequence of regions $X = X_0 , X_1 , \dots , X_n = X'$ such that $X_k$ is adjacent to $X_{k+1}$ for all $k$. Hence there are triangles $\Delta_0, \ldots, \Delta_n$, such that each $\Delta_k$ is adjacent to both $X_k$ and $X_{k+1}$, and by construction this triangle is in $T_{\rho}(K)$. Since the union of triangles in $\widetilde{J}_{\rho} (K, X_k)$ is connected, the triangles $\Delta_k$ and $\Delta_{k+1}$ are connected in $\widetilde{J}_{\rho} (K, X)$, and hence the two edges $e$ and $e'$ are connected by a path in $\Tc$. 
		 
		 Second, let us show that $\mathcal{T} = T_\rho (K)$ is $\rho$-attracting. Let $e = (X,Y,Z ; T , T')$ be an edge that intersects $\Tc$ at a vertex $v = (X,Y) \in \mathcal{E}^{(0)}$, but that is not in $\Tc$ (so $e$ is in the circular neighborhood of $\Tc$). This means that the triangle $\Delta$ adjacent to $X,Y,T$ is in $\Tc$, but the triangle $\Delta'$ adjacent to $X,Y,Z$ is not. 

By construction, one of the regions $X,Y,T$ is in $\Omega_\rho (K)$. If $X$ or $Y$ is in $\Omega_\rho (K)$ then we can assume, without loss of generality, that $\Delta$ is in $\widetilde{J}_\rho (K, X)$ and $\Delta'$ is adjacent to $X$. In that case, the construction of $\widetilde{J}_\rho (K, X)$ ensures that the directed edges of $\Delta'$ are pointing towards $\Delta$. If $X$ and $Y$ are not in $\Omega_\rho (K)$, then we necessarily have $T \in \Omega_\rho (K)$. In that case, we cannot have $|t'| < |t|$, as we would have $T' \in \Omega_\rho (K)$, and this would contradict the connectedness of $\Omega_\rho (K)$. So all the edges in a circular neighborhood of $T$ are pointing towards $\Tc$. 

Now assume that $e$ is an edge belonging to a geodesic from a vertex to $T$ that does not point towards $\Tc$. Then, as the last edge of that geodesic is pointing towards $\Tc$, there is necessarily a fork along that geodesic, which means that one of the regions adjacent to $e$ is in $\Omega_\rho (K)$, and this again contradicts the connectedness of $T_\rho (K)$. 
	\end{proof}
	
	We can now infer directly the following property for Bowditch representations. 
	\begin{Proposition}\label{finite}
		Let $c \in \C$, $[\rho]$ in $\mathfrak{X}_{BQ}(F_2 , \SUtwo)$ and let $M = M(c)$ be the constant defined in (\ref{def:Mc}). For any $K\geq M$, the subgraph $T_\rho(K)$ is finite $($in addition to being connected and $\rho$--attracting$).$
	\end{Proposition}
	
	\begin{proof}
		Given such a representation $\rho$, we have that the set $\Omega(K)$ is finite, and each $X \in \Omega(K)$ is such that $f(\tr (\rho (X))>0$ and the neighbors around the region $X$ have exponential growth. Hence, the subgraph $\widetilde{J}_\rho (K, X)$ is finite. This proves that the subgraph $T_\rho(K)$ is a finite union of finite subgraphs and so is finite.
	\end{proof}

\section{Fibonacci Growth} \label{sec:fibonacci}	

\subsection{Fibonacci functions}
	
	In this section, we give the definition of the Fibonacci function $F_v : \Omega \rightarrow \N$ that is associated to a vertex $v = (\a , \b) \in \mathcal{E}$, and say what it means for a function $g \co \Omega \rightarrow [0 , \infty )$ to have Fibonacci growth.
	
	The map $F_v$ is constructed by induction on the distance $d_v (X)$ from a region $X \in \Omega$ to $v$ which is the minimal graph distance in $\mathcal{E}$ from $v$ to a vertex $v' = (X,Y)$. We have naturally $d_v (\a) = 0$ and $d_v (\b) = 0$. Given any region $Z \neq \a , \b$, there are precisely two regions $X,Y$ such that $X,Y,Z$ are the three regions around a triangle in $\mathcal{E}$ and $d_v (X)< d_v (Z)$ and $d_v (Y) < d_v (Z)$. So we define: 
$$F_v (Z) = \begin{cases} 1 & \mbox{if } d_v (Z) = 0 \\ 
F_v (X) + F_v (Y) & \mbox{if } d_v (Z) \neq 0 \text{ and } X, Y \text{ are as defined above.} \end{cases}$$

	\begin{Definition} Suppose $g \co\Omega \rightarrow [0, \infty)$ and $\Omega' \subset \Omega$. Let $v \in \mathcal{E}$. We say that:
		\begin{itemize}
			\item $g$ has an \emph{upper Fibonacci bound on} $\Omega'$ if there is some constant $\kappa >0$ such that $g(X) \leq \kappa F_v (X)$ for all $X \in \Omega'$,
			\item $g$ has a \emph{lower Fibonacci bound on} $\Omega'$ if there is some constant $\kappa >0$ such that $g(X) \geq \kappa F_v (X)$ for all $X \in \Omega'$,
			\item $g$ has \emph{Fibonacci growth on} $\Omega'$ if it has both upper and lower Fibonacci bound on $\Omega'$.
		\end{itemize}
	\end{Definition}	
	
	The following result will be our basic tool to prove the Fibonacci growth for the logarithm of the norm of trace functions associated to representations in the Bowditch set. This result also shows that the definition of Fibonacci growth is independent of the vertex $v$, since the definition only depends on the asymptotic growth of the function and not on the local behavior. In fact, if in the definition above we would use a different function $F_w$ defined with respect to a different vertex $w$, then the only difference is that we might have to use different constants.
	
	\begin{Lemma}\label{fib_cor_upper}{\rm (Corollary 2.1.2 in Bowditch \cite{bow_mar})}\label{bound_fibonacci}
	Suppose $g \co\Omega \rightarrow [0, \infty)$ satisfies an inequality of the form $g(Z) \le g(X)+g(Y)+c$ $($respectively, $g(Z) \ge g(X)+g(Y)+c)$ for some fixed constant $c$, whenever the regions $X, Y, Z \in \Omega$ meet at a vertex of the trivalent tree $\mathcal{T}$. Then $g$ has lower $($respectively, upper$)$ Fibonacci growth. 
	\end{Lemma}
	
	Another important property of Fibonacci functions is explained by the following result, which relates $F_v$ to the word length of the elements of $\Sc\subset \Gamma/\!\sim$; see  Section \ref{ss:simple} for the definition of $\Gamma/\!\sim$. In fact, we can write out explicit representatives in $\Gamma$ of elements of $\Gamma/\!\sim$ corresponding to given elements of $\Sc$, as we described there. 
	Let $v = (X, Y; T, T') \in \mathcal{E}^{(0)}$. The regions $X, Y \in \Sc$ are represented, respectively, by a pair of free generators $\alpha$ and $\beta$ for $\Gamma$. As explained in Section \ref{ss:simple}, the regions $T$ and $T'$ are represented by $\a\b$ and $\a\b^{-1}$, respectively, and similarly all other regions in $\Omega$ can be written as words in $\a^{\pm 1}$ and $\beta^{\pm 1}$. Note that all the words arising in this way are cyclically reduced. Let $W(\omega)$ denote the minimal cyclically reduced word length of an element $\omega$ with respect to the generating set $\{\a, \b\}$. From this we deduce the following:

	\begin{Proposition}\label{prop:word}
	Suppose $\{\alpha, \beta\}$ is a set of free generators for $\Gamma$ corresponding to regions $X$ and $Y$. Let $v$ be the vertex $v =(X, Y) \in \mathcal{E}^{(0)}$. If $\omega_Z \in \widehat\Omega$ corresponds to a region $Z\in \Omega$, then $W(\omega_Z) = F_v(Z)$. 
	\end{Proposition}
	
\subsection{Upper and Lower Fibonacci growth}
		
For any $[\rho] \in \mathfrak{X}$, we denote by $\phi_\rho \co\Omega \rightarrow \C$ the function $\phi_\rho (X) := \log | \tr (\rho (X)) |$. We consider the function $\phi_\rho^+ := \max \{ \phi_\rho , 0 \}$. The purpose of this section is to prove that, if the representation $[\rho]$ is in the Bowditch set $\mathfrak{X}_{BQ}$, then the map $\phi_\rho^+$ has Fibonacci growth. 

First we show that any representation in $\mathfrak{X}(F_2 , \SUtwo)$ has upper Fibonacci bound. We will use Lemma \ref{bound}. 
	\begin{Proposition}\label{prop:upperFib}
		For any $[\rho]$ in $\mathfrak{X}(F_2 , \SUtwo)$, the function $\phi_\rho^+$ has an upper Fibonacci bound on $\Omega$. 
	\end{Proposition}
	
	\begin{proof}
		Let $v$ be an arbitrary vertex in $\mathcal{E}^{(0)}$. For all regions $Z \in \Omega$ not adjacent to $v$, let $X, Y$ be the regions adjacent to $Z$ such that $d_u (X) < d_u (Z)$ and $d_u (Y) < d_u (Z)$. Applying Lemma \ref{bound} to the vertex $u = (X,Y; Z, T) \in \mathcal{E}^{(0)}$, we obtain that $|z| \leq 5 |x| |y|$ and hence $\phi_\rho^+ (Z) \leq \phi_\rho^+ (X) + \phi_\rho^+ (Y) + \log (5)$. Using Lemma \ref{bound_fibonacci}, we can see that $\phi_\rho^+$ has an upper Fibonacci bound.	
	\end{proof}
		
	The lower Fibonacci bound is not satisfied for a general representation in $\mathfrak{X}(F_2 , \SUtwo)$, but it will be true for representations in the Bowditch set. Recall that a vertex $v \in \mathcal{E}^{(0)}$ corresponds to an edge of the trivalent tree $\mathcal{T}$. Hence if we remove such a vertex from $\mathcal{E}$ we split the graph into two connected components. We will refer to properties satisfied on one such component as occurring ``on one side.'' To prove this, we need the following lemma which shows that if a vertex $v$ is ``attracting on one side'' (meaning that all the geodesics from vertices on one component of $\mathcal{E} \setminus \{v\}$ are pointing towards $v$), then the Fibonacci bound is satisfied for the regions adjacent to that component. 

\begin{Lemma}\label{lem:lowerFib1}
	Let $[\rho] \in \mathfrak{X}(F_2 , \SUtwo)$. Assume that $v$ is $\rho$--attracting on one side, and all the regions adjacent to that component of $\mathcal{E} \setminus \{v\}$ have values greater than $C$. Then there exists $s>0$ such that for all vertices $p \neq v$ in that component, if $X,T,Z'$ are the three regions around $p$ such that $d(Z', v) > \max \{d(v,T) , d(v,X)\}$, then we have:
	$$\log |\phi_\rho^+(Z')| > \log |\phi_\rho^+(X)| + \log |\phi_\rho^+(T)| - s.$$ 
\end{Lemma}

\begin{proof}
The notation is chosen so that the arrow points away from the region $Z'$ and towards the region $Z$. We have the relation $z' = z+ \overline{xt} - x y$. The orientation of the arrow gives that $|z'| > |z|$. So we have: 
	$$|z'| = |z + \overline{xt} - xy| \geq |x| (|t| - |y|) - |z|.$$
	
	Now assume that $|t| > |z|$. This allows us to directly use Lemma \ref{bound}, so that 
	$$2 |z'| > |z'| + |z| > |x| \left( \frac{C}{4} \max\{|x|,|y|\} - |y| \right)  \geq |x| |t| \left( \frac{C-4}{2C} \right).$$ 
This gives  $\log |z'| > \log |x| + \log |t| - \log \left( \frac{C-4}{2C} \right)$.
	
	Note that here we have $|z'|> |y|$, and hence, by induction, if for a given vertex we have $|t|>|z|$, then for all subsequent vertices we also have that inequality.
	
	Finally, we prove that all vertices, except possibly $v$, satisfy that inequality. If $v$ is such that $|t|<|z|$, then the previous lemma states that $|z| > \frac{C}{4} \max\{|x| , |y| \}$. The two vertices at distance $1$ from $v$, which corresponds to $(x,t,z')$ and $(y,t,z''),$ satisfy $|z'|> |y|$ and $|z''| > |x|$, which is the required inequality to apply the previous reasoning to these two vertices.
\end{proof}

We are now going to use Lemma \ref{lem:lowerFib1} to prove the lower Fibonacci bound of representations in $\mathfrak{X}_{BQ}$.

	\begin{Proposition} \label{lower_fib}
		Let $[\rho] \in \mathfrak{X}_{BQ}$. Then the function $\phi_\rho^+$ has lower Fibonacci bound on $\Omega$.
	\end{Proposition}
	
	\begin{proof}
		Let $[\rho] \in \X_{BQ, c}$. We denote $M = M(c)$ and we consider the subgraph $T_\rho (M)$, which is finite, connected, and $\rho$--attracting by Propositions \ref{attracting} and \ref{finite}. By construction, all the vertices on the boundary of $T_\rho (M)$ satisfy the hypothesis of Lemma \ref{lem:lowerFib1} and hence the map $\phi_\rho^+$ has lower Fibonacci growth on each connected component of $\mathcal{E} \setminus T_\rho (M)$. As there are only finitely many such connected components (one for each vertex on the boundary of $T_\rho(M)$), we can then obtain the lower Fibonacci bound on all of $\Omega$.
\end{proof}

\subsection{Characterizations of Bowditch Set} \label{sec:bowditch-characterization}

We can now state different characterizations for the representations in the Bowditch set. Our first characterization is a slight reformulation of the definition. 

\begin{Proposition}
Let $[\rho] \in \mathfrak{X}_c(F_2 , \SUtwo)$. Then $[\rho] \in \X_{BQ}$ if and only if the following two conditions are satisfied:
		\begin{enumerate}
			\item[$(1)$] for all $\gamma \in \Sc$, we have $f(\tr (\gamma)) > 0$,
			\item[$(2)$] for all $K \geq 0$, the set $\Omega_\rho (K)$ is finite.
		\end{enumerate}
\end{Proposition}

\begin{proof}
	For the first condition, Corollary \ref{corol_f} implies that an element $A \in \mathrm{SU} (2, 1)$ is loxodromic if and only if $f(\tr(A)) >0$.
	
	For the second condition, we can see that if $[\rho] \in \mathcal{X}_{BQ, c}$, then $\Omega_\rho(M(c))$ is finite  and hence $\phi_\rho^+$ has lower Fibonacci bound. From this it follows that for all $K \geq M(c)$ we have that the set $\Omega_\rho (K)$ is finite, as we wanted. In addition when $0 \leq K < M(c)$, then we have the inclusion $\Omega_\rho(K) \subset \Omega_\rho(M(c))$. 
\end{proof}

For the next characterization, we use the fact that $\mathrm{SU}(2 , 1)$ can be identified with the group of orientation-preserving isometries of the complex hyperbolic plane, and relate that to the following notion of well-displacing representations. This will be the link with the notion of Bowditch representation developed by Schlich \cite{schlich}.

Recall that for every element $\varphi\in\mathrm{Isom}(\mathbf{H}_\C^2)$ one can consider the displacement $l(\varphi)$, which is defined by: 
$$l(\varphi) := \inf \{ d_{ \mathbf{H}_\C^2}(x , \phi (x)) \mid x \in \mathbf{H}_\C^2 \}.$$
Recall also that for any element $\gamma \in F_2$, we denote by $W(\gamma)$ the minimal cyclically reduced word length with respect to some generating set. In these terms, we now prove the following alternative characterization.

\begin{Proposition}
	A representation $[\rho]$ belongs to $\X_{BQ}$ if and only if there exists $k, m >0$ such that $|l (\rho (\gamma))| \geq k W(\gamma) - m $ for all $\gamma \in \Sc$.
\end{Proposition}

\begin{proof}
For the forward direction, we use the fact that if $[\rho] \in \X_{BQ}$, then Proposition \ref{lower_fib} implies that $\phi_\rho^+$ has lower Fibonacci bound on all of $\Omega$, and hence that there exists $\kappa >0$ such that $\phi_\rho^+(X) \geq \kappa F_v (X)$ for all $X \in \Omega$. Proposition \ref{prop:word} tells us that for the appropriate $v \in \mathcal{E}^{(0)}$ we have that $F_v (\gamma) = W(\gamma)$ for all $\gamma \in \Sc$. We can use the fact that $\mathrm{log}^+(|\tr(A)|) \leq |l(A)|$ for all $A \in \mathrm{SU}(2 , 1)$ to see that for all $\gamma \in \Sc$ we have:
$$|l(\rho(\gamma))| \geq \mathrm{log}^+(|\tr(\rho(\gamma))|) \geq \kappa W(\gamma).$$

Conversely, if we have a representation $\rho$ which satisfies $|l (\rho (\gamma))| \geq k W(\gamma) - m $ for all $\gamma \in \Sc$, then we will show that $\rho$ satisfies conditions $(1)$ and $(2)$ from Definition \ref{def_bow}. Firstly, if we have an element $\gamma \in \Sc$ such that $f(\tr(\rho(\gamma))) \leq 0$, then we can find sequences of regions adjacent to $\gamma$ such that the word length increases, but the length remains bounded. 
Secondly, assume $[\rho] \in \mathfrak{X}_c$. If the set $\Omega(M(c))$ is infinite, then we can see that the inequality $|l (\rho (\gamma))| \geq k W(\gamma) - m$ cannot be satisfied for all $\gamma \in \Sc$, because we only have a finite amount of elements with bounded word length. 
\end{proof}

 \begin{Remark}
 	Note that the proof above shows that if $[\rho] \in \X_{BQ}$, then there exists $k>0$ such that $|l (\rho (\gamma))| \geq k W(\gamma)$ for all $\gamma \in \Sc$. In particular, we do not need the additive constant in the statement of the theorem. On the other hand, the result above is stronger for the converse implication when there is the additive constant, and this is why we are leaving the statement as it is written.
 \end{Remark}

Using this, we can show that all primitive-stable representations are in the Bowditch set. This characterization corresponds to the definition given by Schlich of Bowditch representations, where she proves that Bowditch representations from the free group $F_2$ into isometry groups of Gromov hyperbolic spaces are primitive stable. So a corollary of the characterization above and of Schlich's result is that $\X_{BQ}(F_2 , \SUtwo) = \X_{PS}(F_2 , \SUtwo)$.

\begin{Proposition}	\label{finite_graph}
	Let $c \in \C$ and $K \geq M(c)$, and let $[\rho] \in \X_c$. We have $[\rho]$ is in $\X_{BQ}$ if and only if the subgraph $T_\rho (K)$ is finite.
\end{Proposition}

\begin{proof}
	The forward implication follows from Proposition \ref{finite}, so we only need to prove that if $T_\rho (K)$ is finite for $K \geq M(c)$, then $[\rho] \in \X_{BQ}$.  We prove the contrapositive. Suppose $[\rho] \notin \X_{BQ}$, then we have two cases.
	
	The first case is that there exists $X \in \Omega$ such that $\rho(\gamma)$ is not loxodromic (or equivalently that $f(\tr(\rho(\gamma))) \leq 0$). In this case the arc $\widetilde{J}(K, X)$ (and hence the subgraph $T_\rho (K)$) is infinite.

	The second case is that the set $\Omega(K)$ is infinite, so we can see again that the subgraph $T_\rho (K)$ is infinite because for each $X \in \Omega(K)$ the arc $\widetilde{J}(K, X)$ is non-empty.
	
	Having addressed both cases, we have proved the proposition.
\end{proof}

In the final part of this section we will show why the Bowditch set $\X_{BQ}$ strictly contains the set $\X_{CC}$ of convex cocompact representations in $\X(F_2, \SU(2,1))$. Recall that convex cocompact representations can be defined in many equivalent ways, especially since $\SU(2,1)$ is a rank $1$ Lie group. For example we say that a representation $[\rho]\in \X(F_2, \SU(2,1))$ is \textit{convex cocompact} if there exists $k, m >0$ such that $|l (\rho (\gamma))| \geq k W(\gamma) - m $ for all $\gamma \in F_2$ (or, equivalently, if the (any) orbit map $\tau_\rho$ is a quasi-isometric embedding). We can see from the definition that $\X_{CC} \subset \X_{BQ}.$ We now want to show that this is a proper inclusion. We will use a result of Will \cite{Will-07, Will-12} that describes a 3 dimensional family of discrete, faithful and type-preserving representations of the once-punctured torus in $\X(F_2, \SU(2,1))$. These representations will take simple closed curves to loxodromic elements and will map the commutator to a parabolic unipotent element. Let $F_2:=\langle \alpha, \beta \rangle$. Since simple closed curves in $S_{1,1}$ corresponds to primitive elements of $F_2$, the representations described by Will are in the Bowditch set. One can see this using one of our characterizations of $\X_{BQ}$, or the equivalence of $\X_{BQ}$ with the set of primitive stable representations in $\X(F_2, \SU(2,1))$ implied by Schlich \cite{schlich}. Remember that the peripheral element corresponds to the commutator, which is not a primitive element. The fact that the commutator $[\alpha, \beta]$ is mapped to a unipotent element by Will's representations prevents them from being convex cocompact representations. If one would like to see other examples of representations that are in $\X_{BQ} \setminus \X_{CC}$ and such that the peripheral element is mapped to a parabolic non-unipotent element, one can use the work of Falbel-Parker \cite{Falbel-Parker}, Falbel-Koseleff \cite{Falbel-Kos-rig, Falbel-Kos-circ}, or Parker-Gusevskii \cite{Gus-Parker-03, Gus-Parker-00}. See also Will \cite{Will-16} for a discussion of these references. These examples imply the following result.

\begin{Proposition}
	The set $\X_{BQ}$ strictly contains the set of convex cocompact representations $\X_{CC}$.
\end{Proposition}

\subsection{Openness and Proper Discontinuity} \label{sec:proof}

In this section we prove that the set $\X_{BQ}$ is open and the group $\mathrm{Out}(F_2)$ acts properly discontinuously on it. First we prove the following lemma which will imply the openness of $\X_{BQ, c}$.

\begin{Lemma}
	Let $c \in C$ and let $K \geq M(c)$. For each $[\rho] \in \X_{BQ, c}$, if $T_\rho (K)$ is non-empty, there exists a neighborhood $U_\rho$ of $[\rho]$ in $\X_c$ such that for all $[\sigma] \in U_\rho$, the subgraph $T_\sigma (K) $ is contained in $T_\rho (K)$. 
\end{Lemma}

\begin{proof}
	Let $[\rho] \in \X_{BQ, c}$ and let $K > M(c)$ such that $T_\rho (K)$ is non-empty. First, it is easy to see that for $\sigma$ close enough to $\rho$, the subgraphs $T_\sigma (K)$ and $T_\rho (K)$ have non-empty intersection.
	
	Now consider the set $V \subset \mathcal{E}^{(0)}$ of vertices of $\mathcal{E}$ which are on the boundary of $T_\rho (K)$. Let $v \in V$. We can show that if $[\sigma] \in \X_c$ is close enough to $[\rho]$, then $v$ is also not in the interior of $T_\sigma (K)$. From the definition of $T_\rho (K)$ we can see that if $v = (X, X')$, then we have $| \tr(\rho (X))| > K$  or $| \tr (\rho (X'))| > K$. If we choose $\sigma$ close enough to $\rho$, then we can show that the same inequality will be true for $\sigma$ as well. Since the subgraph $T_\rho (K)$ is finite, there is only a finite number of vertices in $V$, so we can choose $\sigma$ close enough to $\rho$ such that any vertex $v \in V$ is not on the interior of $T_\sigma (K)$. 
	
	Now, since $T_\sigma (K)$ is connected  and $T_\sigma (K) \cap T_\rho (K)$ is non-empty, we have that the subgraph $T_\sigma(K)$ is contained in $T_\rho(K)$. 
\end{proof}

The last result of this paper proves proper discontinuity for the action of $\mathrm{Out}(F_2)$ on $\X_{BQ, c}$. 

\begin{Theorem}
	The set $\X_{BQ, c}$ is open in $\X_c$ and the action of $\mathrm{Out}(F_2)$ is properly discontinuous on $\X_{BQ, c}$.
\end{Theorem}

\begin{proof}
	The previous lemma directly implies openness, using the characterization of representations in $\X_{BQ, c}$ given by Proposition \ref{finite_graph}.

In order to prove that the action of $\mathrm{Out}(F_2)$ is properly discontinuous, let $\mathsf{C}$ be a compact subset of $\X_{BQ, c}$. We want to prove that the set 
	$$\Gamma_{\mathsf{C}} = \{ g \in \mathrm{Out}(F_2) \ | \ g \mathsf{C} \cap \mathsf{C} \neq \emptyset \}$$
	is finite.
	
	Let $K \geq M(c)$ be such that for any $[\rho] \in \mathsf{C}$, the tree $T_\rho (K)$ is non-empty. Now around each element of $\mathsf{C}$ there exists a neighborhood $U_\rho$ given by the previous lemma. So we have that the set $(U_{\sigma})_{[\sigma] \in \mathsf{C}}$ is an open covering of $\mathsf{C}$. Using the compactness of $\mathsf{C}$, we take a finite subcover $(U_{\rho_i})_{i\in I}$, where $I$ is a finite set.
	
	Now, for each element $\rho_i$, we take the tree $T_{\rho_i} (K)$ and consider the union
	$$\mathcal{T}(K) = \bigcup_{i \in I} T_{\rho_i} (K).$$
	By construction, for each element $[\sigma] \in \mathsf{C}$, the subgraph $T_\sigma (K)$ is contained in $\mathcal{T}(K)$. Since the subgraph $\mathcal{T}(K)$ is a finite union of finite subgraphs, it is itself finite. It follows that the set 
	$$\Gamma_1 = \{ g \in \Gamma \ | \ g \mathcal{T}(K)  \cap \mathcal{T}(K)  \neq \emptyset \}$$
	is finite. Finally, as $T_{g\sigma} (K) = g T_\sigma (K)$, we have that $\Gamma_{\mathsf{C}} \subset \Gamma_1$, and hence $\Gamma_{\mathsf{C}}$ is finite, which proves that $\mathrm{Out} (F_2)$ acts on $\X_{BQ, c}$ properly discontinuously.
\end{proof}

\appendix

\section{Computations}

\subsection{\texorpdfstring{$f=0$}{f=0} case}\label{app:fzero}
Recall the set-up from Section \ref{escaping}. Let $\a$ be the region in $\mathcal{C}^{(0)}$ which will serve as our central region. The neighbors are all of the form  $\g_n = \a^n \b$ for all $n \in \N$. Let $[\rho] \in \X_c$ be a representation, with $c \in \C$ (and $P,Q \in \R$). Let $x = \tr (\rho (\a))$ and consider the sequence $u_n = \tr (\rho (\g_n))$. The trace identity gives, for all $n \in \Z$, the following relation: $$u_{n+3} = u_n - \overline{x} u_{n+1} + x u_{n+2}.$$

To study the solutions of this recurrence relation we consider the matrix $$M_x = \begin{pmatrix} x & - \overline x & 1 \\ 1 & 0 & 0 \\ 0 & 1 & 0 \end{pmatrix},$$ and we define $$X_n = \begin{pmatrix} u_{n+2} \\ u_{n+1} \\ u_n \end{pmatrix} \in \mathbb{C}^3,$$ so that we have $X_{n+1} = M_x X_n$.  The matrix $M_x$ is conjugate to a matrix in $\rm{SU} (2,1)$. 

Given a matrix $A\in\rm{SU} (2,1)$, recall that its characteristic polynomial is given by $\chi_A(x) = x^3 - t x^2 + \overline t x -1$, where $t=\tr(A)$. The resultant of $\chi_A$ and its derivative tells us when $\chi_A$ has multiple roots, and hence when $A$ has an eigenvalue with multiplicity greater than $1$. Recall also the definition of the function $f\co\C\to \R$ by $$f(t) = \mathrm{Res} (\chi_A , \chi_A') =  |t|^4 - 8 \mathrm{Re} (t^3) + 18 |t|^2 - 27.$$

\begin{Lemma}
If $f(x)= 0$, then the coefficients of the matrix $M_x^n$ grow at most quadratically in $n$, and hence the growth of the sequence $(X_n)_{n\in \Z}$, and hence of the sequence $(u_n)_{n\in \Z}$ is at most quadratic.
\end{Lemma}

\begin{proof}
Since $f(x)=0$ there are repeated roots to the characteristic polynomial and all roots have modulus $1$.  So, up to conjugation (potentially in $\mathrm{GL}(3,\C)$), the matrix $M_x$ is of the form $$\left(
\begin{array}{ccc}
 a & 1 & 0 \\
 0 & a & 1 \\
 0 & 0 & a \\
\end{array}
\right)\text{ or }\left(
\begin{array}{ccc}
 a & 1 & 0 \\
 0 & a & 0 \\
 0 & 0 & \frac{1}{a^2} \\
\end{array}
\right),$$ where $a$ is a cube root of unity. In the first case, $$M_x^n=\left(
\begin{array}{ccc}
 a^n & na^{n-1} & n\left(\frac{n-1}{2}\right)a^{n-2} \\
 0 & a^n & na^{n-1} \\
 0 & 0 & a^n \\
\end{array}
\right),$$ and so the matrix entries are bounded above by $n\left(\frac{n-1}{2}\right)$.  In the second case, $$M_x^n=\left(
\begin{array}{ccc}
 a^n & na^{n-1} & 0 \\
 0 & a^n & 0 \\
 0 & 0 & a^{-2n} \\
\end{array}
\right),$$ and so the entries are bounded above by $n$.  In both cases the growth is at most quadratic.
\end{proof}

\subsection{\texorpdfstring{$P=6$}{P=6} case}

Recall from earlier that: $$x = \tr (\rho (\a)), \ u_{n+1} = \tr (\rho (\alpha^{n+1} \beta )), \ u_{n+2} = \tr ( \rho (\alpha (\alpha^{n+1}) \beta)), \ \overline{u_n} = \tr (\rho (\alpha (\alpha^{n+1} \beta)^{-1} )).$$  We can identify the quadruple $(x, u_{n+1} , u_{n+2} , \overline{u_n})$ as the character of a representation in $\X_c$, and hence the quadruple will satisfy the equations for $P$ and $Q$. For large $n$, the quadruple is arbitrarily close to a vector of the form $(x, B, B{\bar \lambda}{\lambda^{-1}} , \overline{B}{\bar \lambda}{\lambda^{-1}})$ for $\lambda\in \C^*$ and $B\in \C$.  

Therefore, this quadruple satisfies the  equations for $P$ and $Q$:
	\begin{align*}
		P & = |x|^2 |B|^2 + |x|^2 + |B|^2+|B|^2+|B|^2 - 2 \Re (x B {\lambda}{\bar \lambda}^{-1} \overline{B}) - 2 \Re (\overline{x} B {\bar \lambda}{\lambda^{-1}} \overline{B} ) -2 \\
		& = |B|^2 ( |x|^2 - 4 \Re ( x {\lambda}{ \bar \lambda }^{-1} ) +3) - 3 + |x|^2.
		\end{align*}
So we obtain an expression for $|B|$ in terms of $P$ and $x$:
		$$|B|^2 = \dfrac{P + 3 - |x|^2}{|x|^2 - 4 \Re ( x {\lambda}{\bar \lambda}^{-1} ) +3}.$$

Recall that $x = \lambda + {\bar \lambda}{\lambda^{-1}} + {\bar \lambda}^{-1}$ and $|x|^2 = x \overline{x}$. Letting $P = 6$, and $\lambda=re^{is}$ we derive (using {\it Mathematica}):  

$$|B|^2=1-\frac{2 \left(r^2-1\right)^2}{\left(r^2+1\right) \left(r^2-2 r \cos
   (3 s)+1\right)}.$$

This expression is well-defined for every value of $r$ and $s$ except where $r=1$ and $s=0, 2\pi/3,$ or $4\pi/3$.  In those cases the numerator and denominator are both 0.
   
\begin{figure}
[hbt] \centering
\includegraphics[height=6 cm]{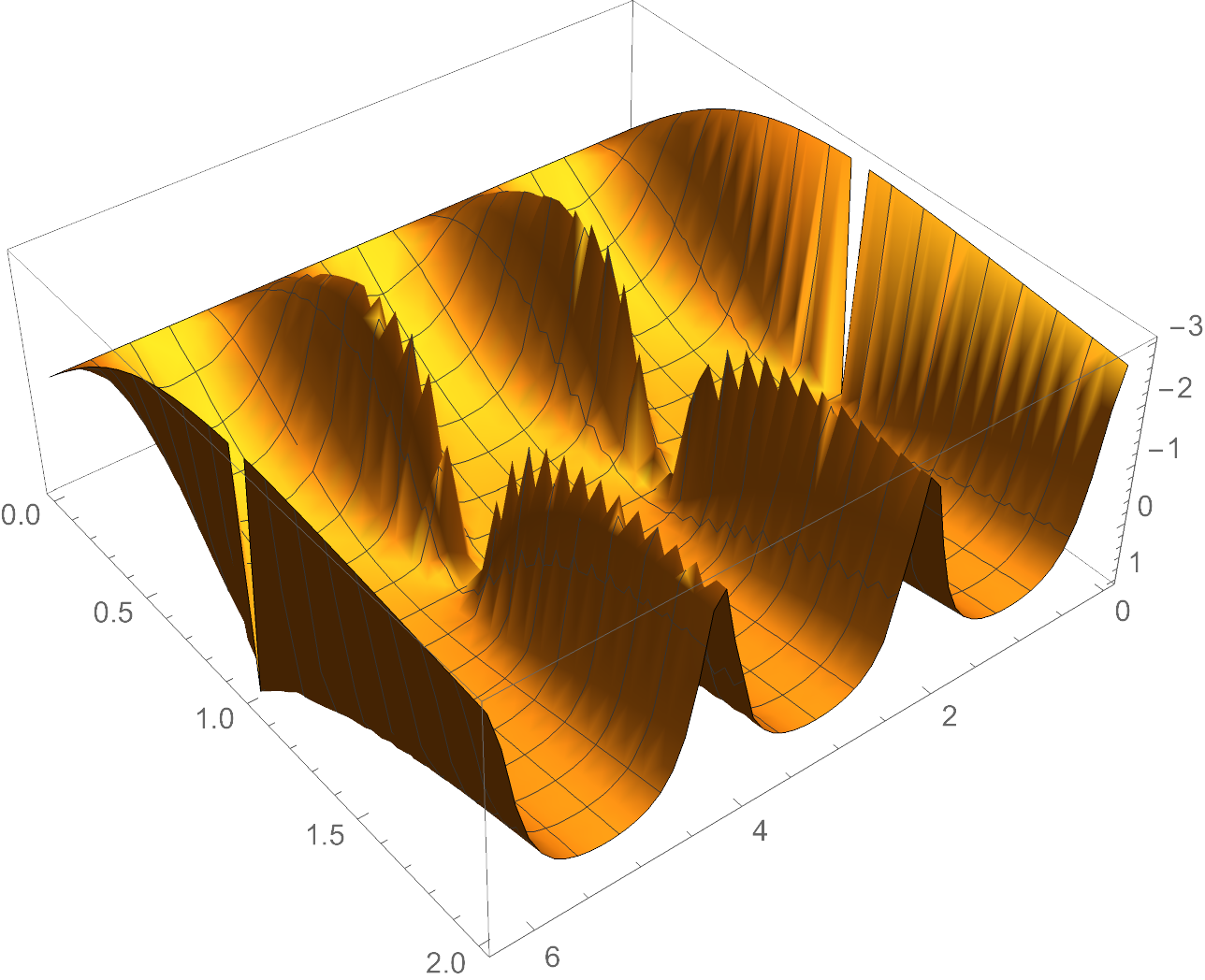}
\caption{Graph of $|B|^2$ when $P=6$.}
\label{fig:sharkfins}
\end{figure}
   
However, for every fixed value of $s$ not equal to $0, 2\pi/3,$ or $4\pi/3$, as $r\to 1$ the limit of $|B|^2$ is 1.  But for $s$ equal to any of $0, 2\pi/3, 4\pi/3$, the limit as $r\to 1$ is $-3$.  So when $\lambda$ converges to an element of $\zeta_3$, despite both the numerator and the denominator going to $0$, the expression stays bounded (see Figure \ref{fig:sharkfins}).  If this was not the case, then we would {\it not} have the situation that, for all fixed value of $s$ not equal to $0, 2\pi/3,$ or $4\pi/3$, as $r\to 1$, the limit of $|B|^2$ is $1$.

\subsection{\texorpdfstring{$P>6$}{P>6} case}

If $P > 6$, then the expression for $|B|^2$ as a function of $\lambda$ is unbounded when $\lambda$ gets close to $1$. In that case, we also need to use the expression for $Q$ which, after simplification, is:

\begin{align*}
Q = & |B|^4 \left(5 + |x|^2 + 2 \Re \left( 2 (\frac{\bar \lambda}{\lambda})^3- 2 \overline x  \frac{\bar \lambda}{\lambda} - x  (\frac{\bar \lambda}{\lambda})^2 \right) \right) \\
& + |B|^2 \left( |x|^4+3|x|^2-18+2\Re \left( 2 x^2 \frac{\bar \lambda}{\lambda}-3 x  (\frac{\bar \lambda}{\lambda})^2+6 \overline x  \frac{\bar \lambda}{\lambda} - 2 |x|^2 \overline x  \frac{\bar \lambda}{\lambda}+\overline x^2( \frac{\bar \lambda}{\lambda})^2-x^3 \right) \right) \\
& + 9 - 6|x|^2+2\Re (x^3).
\end{align*}

Replacing the expression of $|B|^2$ in terms of $P$ and $\lambda$ found previously and substituting  $\lambda=re^{is}$, we have an expression relating $Q, P$ and $\lambda$ (derived using {\it Mathematica}):
    
$$Q_P(\lambda)=\frac{\left(P r^4-r^6+1\right) \left(Pr^2+r^6-1\right)}{r^4 \left(r^2+1\right)^2}.$$
   
Amazingly $Q_P(\lambda)$ does not depend on $s$ at all.  Letting $r\to 1$ we see $Q_P (\lambda) \to \frac{P^2}{4}$.

\end{document}